\documentclass[11pt, a4paper]{amsart}
\usepackage{amsmath, amsthm, amssymb, amsfonts, enumerate}
\usepackage[colorlinks=true,linkcolor=blue,urlcolor=blue]{hyperref}
\usepackage{verbatim}
\usepackage[dvips]{color}
\usepackage[active]{srcltx}
\usepackage{mathrsfs}
\usepackage{latexsym}
\usepackage{graphicx}
\usepackage{float}
\usepackage{mathrsfs}
\usepackage[active]{srcltx}
\usepackage{xcolor}
\usepackage{geometry}\usepackage{epstopdf}
\usepackage{commath}
\usepackage{caption}
\usepackage{accents}
\usepackage{soul}

\allowdisplaybreaks

\newtheorem{theorem}{Theorem}[section]

\newtheorem{proposition}{Proposition}[section]
\newtheorem{lemma}{Lemma}[section]
\newtheorem{corollary}{Corollary}[section]
\theoremstyle{definition}

\newtheorem{remark}{Remark}[section]

\floatstyle{plain} \restylefloat{figure} \restylefloat{table}

\def\A{\mathcal A}
\def\B{\mathcal B}
\def\C{\mathcal C}

\def\E{\mathbb E}
\def\F{\mathcal F}

\def\L{\mathcal L}

\def\P{\mathbb P}

\def\R{\mathbb R}
\def\S{\mathcal S}

\def\ud{\mathrm d}
\def\eps{\varepsilon}
\def\ep{\varepsilon}

\newcommand{\kom}[1]{}
\renewcommand{\kom}[1]{{\bf [#1]}}
\definecolor{blau}{rgb}{0.1,0.0,0.9}

\newcounter{komcounter}
\numberwithin{komcounter}{section}

\title[Maximality principle in the dividend problem]{The maximality principle in singular control with absorption and its applications to the dividend problem}
\author[T.\ De Angelis, E.\ Ekstr\"om, M.\ Olofsson]{Tiziano De Angelis, Erik Ekstr\"om and Marcus Olofsson}
\address{T.\ De Angelis: School of Management and Economics, Dept.\ ESOMAS, University of Torino, Corso Unione Sovietica, 218 Bis, 10134, Torino, Italy; Collegio Carlo Alberto, Piazza Arbarello 8, 10122, Torino, Italy.}
\email{\href{mailto:tiziano.deangelis@unito.it}{tiziano.deangelis@unito.it}}
\address{E.\ Ekstr\"om: Department of Mathematics, Uppsala University, Box 256, 75105 Uppsala, Sweden.}
\email{\href{mailto: ekstrom@math.uu.se}{ekstrom@math.uu.se}}
\address{M.\ Olofsson: Department of Mathematics and Mathematical Statistics, Ume\r a University, 90187 Ume\r a, Sweden}
\email{\href{marcus.olofsson@umu.se }{marcus.olofsson@umu.se }}

\date{\today}

\begin{document}

\maketitle

\begin{abstract}
Motivated by a new formulation of the classical dividend problem, {we show that Peskir's}
{\em maximality principle} {can be transferred to} singular stochastic control problems with 2-dimensional degenerate dynamics and absorption along the diagonal of the state space.  
We construct an optimal control as a Skorokhod reflection along a moving barrier, where the barrier can be computed analytically as the smallest solution to a certain 
non-linear
ordinary differential equation. Contrarily to the classical 1-dimensional formulation of the dividend problem, our framework produces a non-trivial solution when the firm's 
(pre-dividend) equity 
capital evolves as a geometric Brownian motion. Such solution is also qualitatively different from the one traditionally obtained for the arithmetic Brownian motion.
\end{abstract}

\section{Introduction}
The modern formulation of De Finetti's classical {\em dividend problem} \cite{DeF} is a very popular example of a singular stochastic control (SSC) problem with absorption of the state dynamics. The absorption feature captures the default of a firm whose capital evolves randomly in time and that pays dividends to its share-holders according to a singular control strategy that must be determined via a stochastic optimisation. Another application of SSC with absorption can be found in the literature on optimal resource extraction under stochastic fluctuations. An early contribution in that area is a problem of optimal harvesting of a population formulated and solved by Alvarez and Shepp \cite{AS}, where the absorption describes the extinction of the population being harvested. The basic idea in this class of problems is that exerting control may endogenously trigger absorption of the state-process, which is generally undesirable. Therefore, when constructing optimal strategies one needs to find a trade-off between exerting control (i.e., paying dividends or harvesting) and keeping a sufficiently high reserve (cash or resources) to withstand future fluctuations in the dynamics.

Mathematically, SSC problems with absorption are harder to study than their counterpart without absorption. This is due to the fact that the absorption feature introduces an inhomogeneity in the state space that translates into additional boundary conditions in the Hamilton-Jacobi-Bellman equation associated to the stochastic control problem. When the underlying dynamics is 1-dimensional, an approach based on an educated guess for the optimal strategy and a verification theorem (so-called {\em guess-and-verify}) is generally adopted to obtain solutions in closed-form. 
In higher dimensions, guessing-and-verifying is not always feasible.
However, some two-dimensional stochastic control problems with {\em degenerate dynamics} are known to be tractable and produce solutions in closed form (yet not explicit, in general). Notably, in this class of problems we find Markovian stopping problems where the payoff upon stopping depends on the supremum process (cf. \cite{DSS}).
Such considerations motivate our study of SSC with two-dimensional degenerate dynamics and absorption. In this context we develop 
a {solution} method
that transfers the so-called {\em maximality principle} in optimal stopping (Peskir \cite{P98}) to SSC.

For the ease of presentation we focus on 
a variant of the classical
dividend problem; extensions beyond this model are possible and they are highlighted in Remark \ref{rem:extensions}. A control (or \textit{dividend strategy}) is a non-decreasing stochastic process $D$ that stands for the cumulative amount of dividends paid by a firm to its share-holders over time. Denoting by $\gamma$ the firm's default time, the dividend problem can be stated informally as
$$
\mbox{
Find $D$ that maximizes $\E  \left[\int_{0}^{\gamma} e^{-rt} \ud D_t\right]$.
}
$$

A common approach in the literature is to use a diffusion approximation for the firm's net capital. The capital may fluctuate because of gains and losses incurred by the firm over time and the traditional example is that of an insurance company that collects premia at a certain rate and pays claims as and when they occur. In fact, a benchmark in the literature is to model the (post-dividend) equity
capital as a Brownian motion with drift subject to a downward push, i.e., 
$$
Y_t^D= y+\mu t + \sigma W_t - D_t.
$$
In this setting the default time $\gamma$ is the first time $Y^D$ goes below $0$. It has been shown (see Asmussen and Taksar \cite{AT}, Jeanblanc and Shiryaev \cite{JPS}, Radner and Shepp \cite{RS}) that the optimal strategy is of threshold type, i.e., it is optimal to pay the minimal amount of dividends required to ensure that $Y^D$ stays below a constant threshold $b$, which can be determined explicitly.
The constant coefficient case admits two natural interpretations: 
\begin{itemize}
\item[(i)] $Y^D$ represents the  {\em post}-dividend 
{equity}
capital of a company, i.e., the holdings after dividend payments have been deducted according to a strategy $D$ (as described above);
\end{itemize}
or
\begin{itemize}
\item[(ii)] an arithmetic Brownian motion $Y_t=y+\mu t+\sigma W_t$ models the firm's  {\em pre}-dividend equity
capital, i.e., the equity capital that the firm would have if no dividends were ever paid out, while $D$ is a given 
dividend strategy. In this case the default time links the two processes via the relationship $\gamma=\inf\{t\geq 0:Y_t\leq D_t\}$.
\end{itemize}
For constant coefficients, the two formulations are equivalent (set $Y^D=Y-D$). Instead, 
when generalising to an underlying process that follows a 1-dimensional diffusion with state-dependent coefficients $\mu(\cdot)$ and $\sigma(\cdot)$, the two settings are truly different: in particular, either the coefficients depend on the {\em post}-dividend equity
capital, or on the {\em pre}-dividend equity
capital (or, in a more refined model, on both).
In the first case, the process $Y^D$ and absorption time $\gamma$ are defined as
\begin{equation} \label{eq:formulation1}
Y^D_t =  y + \int_0 ^t \mu(Y^D_s)  \ud s + \int_ 0^t \sigma(Y^D_s) \ud W_s - D_t
\end{equation}
and
\[
\gamma=\gamma^D =\inf\{t \geq 0 : Y^D_t \leq 0\},
\]
respectively. In the second case instead the 
pre-dividend equity capital
evolves as an uncontrolled process
\begin{equation} \label{eq:formulation2}
Y_t =  y + \int_0 ^t \mu(Y_s)  \ud s + \int_ 0^t \sigma(Y_s) \ud W_s
\end{equation}
and
\[
\gamma =\gamma^D=\inf\{t \geq 0 : Y_t \leq D_t\}.
\]
The first formulation \eqref{eq:formulation1} is 
well-suited for problems of resource extraction, where the rate of reproduction depends on the current population size. The problem is
one-dimensional
in the sense that a sufficient statistics consist of only the current level of $Y^D$. As a consequence, the value function of the problem is characterised by a free-boundary problem in terms of an ordinary differential equation (ODE) (e.g., Shreve, Lehoczky, Gaver \cite{SLG}).  
In contrast, the second formulation
\eqref{eq:formulation2} has a two-dimensional sufficient statistic $(D,Y)$ and the associated free-boundary problem is therefore more involved. 

In the current article, we study the two-dimensional formulation \eqref{eq:formulation2}. From a financial perspective that model assumes that the law of the firm's (pre-dividend) equity
capital $Y_{t+\ud t}$ at time $t+\ud t$ depends on its own value 
$Y_t$ at time $t$ via the coefficients in the stochastic differential equation (SDE) in \eqref{eq:formulation2}, but it does not depend on the amount of dividends paid to share-holders. However, the actual cash reserve (the post-dividend equity 
capital) of the firm at any time $t$ is given by the difference $Y_t-D_t$ and, over time, the firm cannot pay out in dividends more than its total (pre-dividend) equity
capital.

We find in this paper that the conditions for non-trivial solutions in the two cases \eqref{eq:formulation1} and \eqref{eq:formulation2} differ considerably. Notably, it is well-known that the standard financial model using a geometric Brownian motion (gBm) is degenerate in the first formulation: if the drift exceeds the discount rate (in the notation of Section~\ref{sec8} below, $\alpha>r$), then the value is infinite; if instead the drift is smaller than the discount rate ($\alpha\le r$), then it is optimal to distribute an initial lump sum payment of size $y$ that leads to immediate default of the firm (absorption). In contrast, the second formulation gives rise to a non-degenerate problem,
the details of which are provided in Section~\ref{sec8} below.

Our main contributions are threefold:

\begin{itemize}
\item[(i)]
We study a new formulation of the dividend problem. We establish conditions under which its solution is given by a dividend strategy of (stochastic) moving-barrier type and we obtain the barrier level as minimal solution of an associated non-linear ODE. Our formulation covers standard financial models building upon gBm that produce optimal strategies that are qualitatively and quantitatively different from the classical models with arithmetic Brownian motion (see Remark \ref{rem:quality}).
\item[(ii)] We show that Peskir's {\em maximality principle} \cite{P98} for optimal stopping problems involving the supremum process finds applications in the context of our SSC problems with absorption.
Although our results are presented for the dividend problem, the methods and the maximality principle can be adapted to more general situations at the cost of dealing with more involved ODEs for the optimal barrier. That, however, leads to potentially difficult questions about existence of a minimal solution of such ODEs.
\item[(iii)]
We are able to transfer the maximality principle from optimal stopping to SSC
by extending the well-known connection between singular control and optimal stopping (see 
Bather and Chernoff \cite{BC}, Baldursson and Karatzas \cite{BKar}, Boetius and Kohlmann \cite{BK}, Karatzas and Shreve \cite{KS}) to the current case of 
{\em{two-dimensional}}
singular control 
{\em with absorption}. 
The derivative of the value function in the dividend problem with respect to the state variable associated to the process $D$ is the value function of an optimal stopping problem for a two-dimensional degenerate diffusion {\em with oblique reflection}
at the diagonal of the first quadrant in the Cartesian plane. The gain function depends on such dynamics via a state-dependent exponential factor which increases upon each reflection at a `rate' depending (informally) on the `local-time' of the process at the diagonal. We emphasise that the original connection between singular control and optimal stopping (see \cite{BC}, \cite{BKar}, \cite{BK}, \cite{KS}) has a {\em different structure} compared to ours. In those papers, the controlled dynamics does not undergo absorption and, as a result, the optimal stopping problem does not involve reflecting processes and local times. The mathematical arguments that provide the connection in \cite{BC}, \cite{BKar}, \cite{BK}, \cite{KS} do not apply to our setting as they rest on convexity/concavity of the expected payoff of the SSC problem with respect to the initial value of the controlled state variable. {That condition} breaks down in our framework because of the additional absorption (default) time $\gamma$. When convexity/concavity are not in place, as in our setting, a connection between SSC and optimal stopping cannot be taken for granted. Indeed, it was shown in De Angelis et al.\ \cite[Sec.\ 3]{DeAFM15} that without convexity/concavity the classical connection in the spirit of Bather and Chernoff \cite{BC} fails.
\end{itemize}

The paper is organised as follows.
In Section \ref{sec:setting} we present a detailed problem formulation, and we state our main result (Theorem \ref{thm:mainthm}). The theorem derives  an optimal dividend strategy transferring the maximality principle from optimal stopping problems to singular control problems with absorption.
Section \ref{sec6} presents the key heuristic ideas that led us to the derivation of 
the solution of the singular control problem. 
Sections \ref{sec3}--\ref{sec4} are devoted to the proof of Theorem \ref{thm:mainthm}. In Section \ref{sec5} we establish a connection between our SSC problem with absorption and an optimal stopping problem, highlighting the link between the 
maximality principle and SSC. In Section \ref{sec8} we apply our main result to solve our version of the dividend problem for 
gBm.

\section{Setting and main results}\label{sec:setting}

In this section we first formulate the stochastic control problem and define its value function. Then we introduce a class of solutions of a certain ODE and we associate with it a collection of candidate value functions for the control problem. Finally, we construct suitable admissible controls (via Skorokhod reflection) and we use them to state our main result (Theorem \ref{thm:mainthm}).

\subsection{Problem formulation}
Throughout the paper we consider a filtered probability space $(\Omega, \F, (\F_t)_{t\ge 0}, \P)$ equipped with a Brownian motion $W:=(W_t)_{t\ge 0}$ adapted to $(\F_t)_{t\ge 0}$. The filtration is augmented with $\P$-null sets and it is right-continuous. We denote by $Y$ the unique strong solution on $[0,\infty)$ to
\begin{equation}
\label{Y}
Y_t =y+\int_0^t \mu (Y_s) \ud s + \int_0^t \sigma(Y_s) \ud W_s\,,
\end{equation}
where $y\geq 0$ and $\mu:[0,\infty)\to [0,\infty)$, $\sigma : [0,\infty) \to [0,\infty)$ are locally Lipschitz continuous functions with at most linear growth on $(0,\infty)$, with $\sigma(y)>0$ for $y>0$. 
The process $Y$ is {\em regular} in the sense that it visits each point of $(0,\infty)$ in finite time with positive probability (provided $y>0$).  We further assume that $0$ is an absorbing boundary point in case it can be reached in finite time, 
and that $\infty$ is a {\em natural} boundary point so that $Y$ does not explode in finite time.

For a fixed starting point $(x,y)$ with $0\leq x\leq y$,
alongside the process $Y$
we consider the purely controlled dynamics
\[X^D_t:=x+D_t ,\]
where $D$ is a non-decreasing, right-continuous and $(\F_t)$-adapted process with $D_{0-}=0$.
For a fixed process $D$ and any initial point $(x,y)$ with $0\leq x\leq y$ we let 
\begin{align}\label{eq:gamma}
\gamma:=\gamma_{x,y}(D):=\inf\{t\ge 0 : Y_t\le X^D_t\}
\end{align}
and say that $D$ is {\em admissible} if, in addition, $X^D_{\gamma_{x,y}(D)}\leq Y_{\gamma_{x,y}(D)}$ (notice that this implies $X^D_{\gamma}= Y_{\gamma}$~a.s.). In other words, the process $X$ cannot jump strictly above the process $Y$. We then denote the class of admissible controls by 
\begin{align}\label{eq:Dxy}
\A_{x,y}:=\{D\,:&\,t\mapsto D_t\:\text{is non-decreasing, right-continuous, $(\F_t)$-adapted,}\\ \notag
&\mbox{with $D_{0-}=0$ and }X^D_{\gamma_{x,y}(D)}\leq Y_{\gamma_{x,y}(D)}\}.
\end{align} 

In the problem formulation we find it convenient to use the notations $\P_{x,y}(\,\cdot\,):=\P(\,\cdot\,|X_{0-}=x,Y_{0}=y)$ and $\E_{x,y}[\,\cdot\,]:=\E[\,\cdot\,|X_{0-}=x,Y_{0}=y]$. 
For any $y\ge x\ge 0$ and an arbitrary $D\in\A_{x,y}$, the objective function in our stochastic control problem reads
\begin{align*}
J(x,y;D) := \E_{x,y}\Big[ \int_{[0,\gamma]} e^{-r s} \ud D_s\Big],
\end{align*}
where the integral $\int_{[0,\gamma]}$ is in the 
Lebesgue-Stieltjes
sense, including atoms of the (random) measure $t\mapsto \ud D_t(\omega)$ at times 0 and $\gamma(\omega)$. 
The value function of our problem is then defined as
\begin{equation} \label{eq:dividendvalue}
V(x,y) = \sup_{D \in \A_{x,y}} J(x,y;D).
\end{equation}

\begin{remark}
Problem \eqref{eq:dividendvalue} is a two-dimensional singular stochastic control problem with absorption occurring at the first time the underlying controlled process $(X^D,Y)$ hits the diagonal $\{(x,y):x=y\}$. 
The problem is {\em degenerate} since there is no diffusion in the direction of the controlled dynamics $X^D$. 
\end{remark}

\subsection{A class of solutions for an ODE}

For an arbitrary $c\in(0,\infty)$ the {\em scale function} $S(y)$ of $Y$ reads 
\begin{align}\label{eq:S}
S(y)=\int_{c}^y \exp\left(-\int_{c}^z\frac{2\mu(\lambda)}{\sigma^2(\lambda)}\ud \lambda\right)\ud z.
\end{align}
The infinitesimal generator $\L$ of the process $Y$ killed at a rate $r$ is defined by its action on functions $f\in C^2([0,\infty))$ as
\begin{equation}
\label{op}
\L f (y) = \frac{\sigma^2(y)}{2} f_{yy}(y) + \mu(y) f_y(y) -r f(y).
\end{equation}
Denote by $\psi$ and $\varphi$ two solutions of the ODE $\L f=0$ on $(0,\infty)$ such that $\varphi$
is positive and strictly decreasing and $\psi$ is positive (on $(0,\infty)$) and strictly increasing with 
$\varphi(\infty)=0$ and $\psi(\infty)=\infty$. These functions can be chosen as the fundamental solutions of \eqref{op} and are then unique up to multiplication by a positive constant, if appropriate boundary conditions are imposed at 0 (cf. \cite[Chapter II]{BS}). 
It is then known (see e.g. \cite{BS}), and also easy to verify using $\L\varphi=\L\psi=0$, that 
$S'(y)=C(\varphi(y)\psi'(y)-\psi(y)\varphi'(y))$
for some constant $C>0$. For simplicity, and with no loss of generality, we assume that constants are chosen so that 
\[S'(y)=\varphi(y)\psi'(y)-\psi(y)\varphi'(y).\]

Now let
\begin{align}\label{eq:F}
F(x,y):=  \frac{\sigma^2(y)}{r\delta(x,y)} \bigg[ \Big( \varphi'(y) \psi'(x)-\varphi'(x)\psi'(y)\Big) + \frac{\mu(x)}{\sigma^2(x)} \Big( \varphi'(y) \psi(x)-\varphi(x)\psi'(y)\Big) \bigg] 
\end{align}
for $0<x<y$,
with 
\[\delta(x,y):= \varphi(x)\psi(y) -\varphi(y)\psi(x).\] 
Notice that for $x<y$ we have
\begin{align}\label{eq:gammaneg}
\delta(x,y) > 0
\end{align}
by the strict monotonicity and positivity of $\psi$ and $\varphi$, so the denominator in $F$ is well-defined. Next, consider the nonlinear ODE
\begin{equation} \label{eq:ODE}
b'(x)=F(x,b(x)),\quad x> 0.
\end{equation}
We do not specify an initial datum for the ODE but instead look at solutions from the class 
\begin{align}\label{eq:Bcirc}
\begin{aligned}
\B:=\big\{b\in C^1((0,\infty)) \cap C([0,\infty)):&\,\text{$b$ is a solution of \eqref{eq:ODE}}\\ 
&\,\text{with $F(x,b(x))>0$ and $b(x)>x$ for $x> 0$}\big\}.
\end{aligned}
\end{align}
Notice that as part of the definition of $\B$ we require that it only contains solutions of \eqref{eq:ODE} that do not explode for finite values of $x\in[0,\infty)$.

Given $b\in\B$, the inverse function $b^{-1}:[b(0),\infty)\to[0,\infty)$ is well-defined and 
strictly increasing and when $b(0)>0$ we extend the definition by setting $b^{-1}(y)\equiv 0$ for $y\in[0,b(0))$.
Notice that $b^{-1}\in C([b(0),\infty))\cap C^1((b(0),\infty))$ with 
\[
(b^{-1})'(y)=\frac{1}{F\big(b^{-1}(y),y\big)},\quad \text{for $y> b(0)$.}
\]
We associate with $b\in\B$ a function $v^b$ defined by $v^b(0,0)=0$ and for $y>0$ by
\begin{align}\label{eq:Uc}
v^b(x,y)=
\left\{
\begin{array}{ll}
\bar v^b(x,y), & x\leq y\leq b(x),\\[+5pt]
\bar v^b\big(b^{-1}(y),y\big)+b^{-1}(y)-x, & y> b(x),  
\end{array}
\right.
\end{align}
where
\begin{align}\label{eq:vbbar}
\bar v^b(x,y):=\varphi(y)\int_x^y \frac{\psi'\big(b(z)\big)}{S'\big(b(z)\big)} \ud z-\psi(y)\int_x^y \frac{\varphi'\big(b(z)\big)}{S'\big(b(z)\big)} \ud z.
\end{align}
Since $S'>0$ on $(0,\infty)$ then the integrals in the definition of $\bar v^b(0,y)$ are well-defined for $x\leq y\leq b(x)$ with $y>0$. 

We will show below (Theorem~\ref{thm:mainthm}) that for a certain choice of $b^*\in\B$, the function $v^{b^*}$ coincides with the value function $V$ from \eqref{eq:dividendvalue}. 
The expression for $v^b$ may seem a bit {\em ad-hoc} at the moment but it will be fully motivated in Section \ref{sec:heuristics} below.

\subsection{Solution of the stochastic control problem} We start by stating a lemma for the construction of a suitable class of admissible controls. 
The proof is standard and we provide it in the Appendix for completeness.

\begin{lemma}\label{lem:SK}
Let $y\geq x\geq 0$.
For an arbitrary $b\in\B$, set 
\begin{equation}
D^b_t = \sup _{0\le s\le t}(b^{-1}(Y_s)-x)^+,\quad\P_{x,y}-a.s.
\end{equation}
Then $D^b\in\A_{x,y}$.  
Moreover, letting $X^b_t:=x+D^b_t$ and $\gamma^b:=\inf\{t\ge0\,:\,Y_t\le X^b_t\}$, the pair $(X^b,Y)$ solves the Skorokhod reflection problem
\begin{equation}\label{eq:SKsol}
X^b_{t\wedge\gamma^b}\ge b^{-1}(Y_{t\wedge\gamma^b})\quad\quad\text{and} 
\quad\quad\int_{[0,\,t\wedge\gamma^b]}1_{\{X^b_{s}>b^{-1}(Y_s)\}}\ud D^b_s
=0,
\end{equation}
for all $t\geq 0$, $\P_{x,y}$-a.s.
\end{lemma}
\vspace{+3pt}

We now present the main result of the paper, which is the characterisation of the optimal control in our optimisation problem \eqref{eq:dividendvalue} via the maximality principle. 
\begin{theorem}\label{thm:mainthm}
Fix $b^*\in\B$, and let $v^*(x,y):=v^{b^*}(x,y)$ as in \eqref{eq:Uc}.
Set $D^*:=D^{b^*}$, $X^*:=X^{b^*}$ and $\gamma^*:=\gamma^{b^*}$ as in Lemma~\ref{lem:SK}. If
\begin{equation}\label{eq:integr}
\P_{x,y}(\gamma^*<\infty)=1 \quad\text{and}\quad \E_{x,y}\left[\sup_{t\geq 0}\left\{e^{-rt}v^*(x,Y_t)\right\}\right]< \infty
\end{equation}
for $0< x\le y$, 
then $D^*$ is an optimal control in \eqref{eq:dividendvalue} and $v^*$ coincides with $V$, i.e.,
\begin{align}\label{eq:vV}
V(x,y)=J(x,y;D^*)=v^*(x,y)
\end{align} 
for $0< x\le y$.
Moreover,
\begin{equation}
\label{inf}
b^*(y)=\inf_{b\in\B}b(y).
\end{equation}
\end{theorem}

The proof of Theorem \ref{thm:mainthm} and further properties of the value function are presented in Sections \ref{sec:proof} and \ref{sec:max} below. 

\begin{remark}\label{rem:mainth}
In Section~\ref{sec:max} we specify general conditions under which solutions of \eqref{eq:ODE} in $\B$ exist and are ordered, 
so that $b^*$ is the minimal element of $\B$ as stated in \eqref{inf}. We notice that Peskir \cite{P98} works with boundaries that are below the diagonal, so `maximal' in his setting and `minimal' in our setting are equivalent notions.
\end{remark}

\begin{figure}
\includegraphics[scale=0.2]{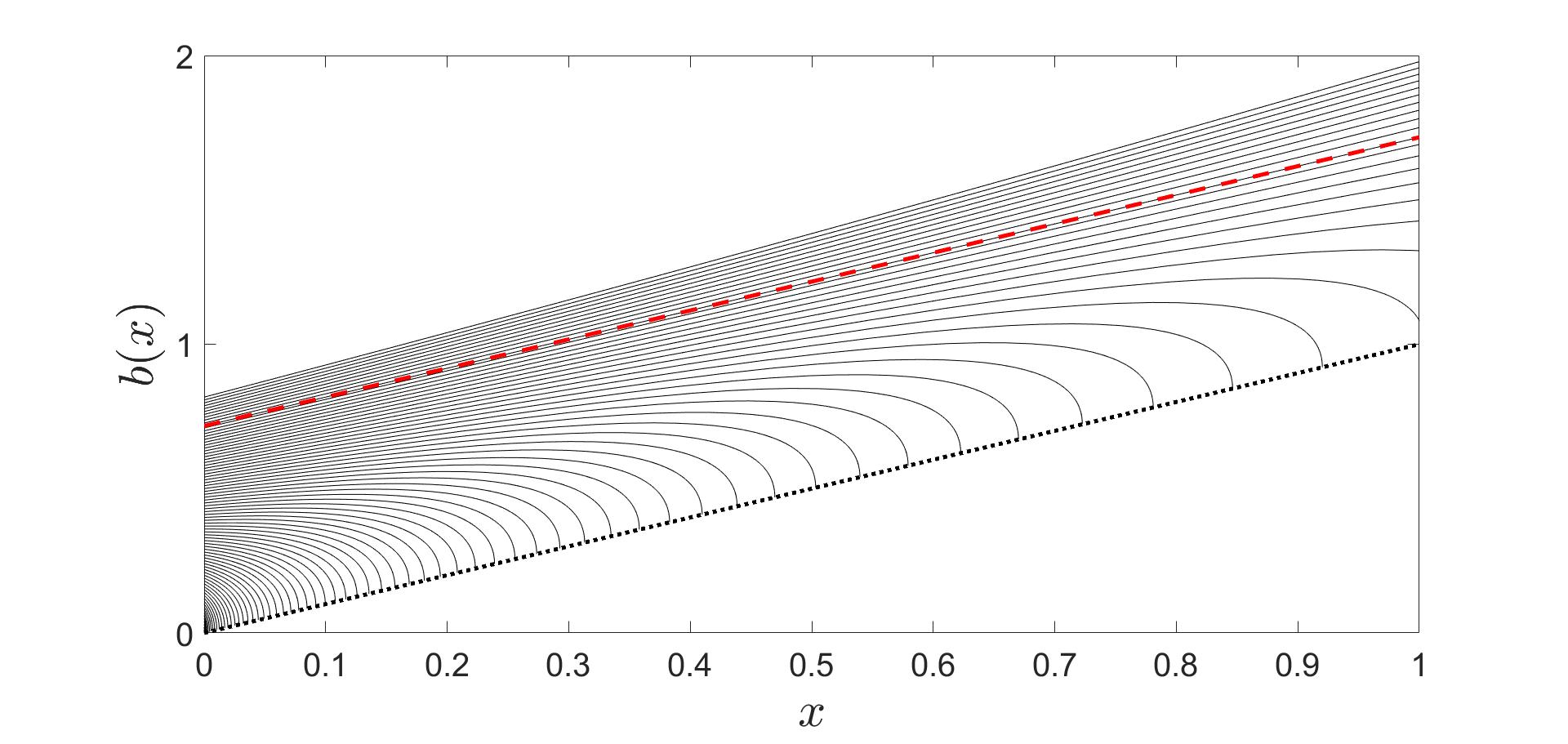}
\caption{Solutions to the ODE \eqref{eq:ODE} for varying initial data using $r=0.05$ and constant coefficients $\mu \equiv 0.04 $ and $\sigma\equiv 0.3$. In this setting, the optimal dividend strategy is described by an affine boundary $b^\ast$ (dashed line). All solutions below $b^\ast$ hit the diagonal (dotted line). Solutions are obtained using MATLABs ODE-solver \texttt{ode15s}.}
\label{fig1}
\end{figure}

\begin{remark}
It will be clarified in Corollaries \ref{rem:unique} and \ref{cor:bb} that there can be at most one function $b^*\in\B$ that satisfies the integrability conditions \eqref{eq:integr}. Hence, it must be the minimal element in $\B$ because any other element $b\in\B$ is associated to a larger $\gamma^b$ by continuity of paths of $(X^b,Y)$. 
\end{remark}

%%%%%%%%%%%%%%%%%%%%%%%%%%%%%%%%%%%%%%%%%%%%%%

\section{Heuristic derivation of the variational problem} \label{sec:heuristics}
\label{sec6}

The construction of our solution to the singular control problem via the {\em maximality principle} in Theorem \ref{thm:mainthm} can be derived from heuristic ideas that we illustrate in this section.

In line with the literature on the dividend problem, it is intuitively clear that control should be exerted when $Y$ is sufficiently bigger than $X$, so that the risk of bankruptcy remains small. At the same time, waiting is penalised by discounting the future payoff, so that it would not be optimal to wait indefinitely for ever larger values of $Y$. In contrast with the classical set-up, where dividend payments affect directly the diffusive dynamics $Y$ of the firm's equity capital, here the decision to make a dividend payment should depend on the amount of dividends that have already been paid. So, letting $x>0$ denote the total amount of dividends paid so far, we expect that there should exist a critical value $b(x)>x$ such that no further dividends are paid at times $t\ge 0$ such that $x<Y_t<b(x)$. 

As long as it is optimal to pay no dividends, the discounted value of the problem should remain constant on average, i.e., $t\mapsto e^{-rt}V(x,Y_t)$ should be a {\em martingale} for as long as $x<Y_t<b(x)$. Moreover, if an amount $\delta>0$ of control is used at time zero, the resulting payoff is at most $\delta+V(x+\delta,y)$ and, in general, one has $\delta+V(x+\delta,y)\le V(x,y)$. Dividing by $\delta$ and letting $\delta\to 0$, we expect that if $x<y<b(x)$ then $V_x(x,y)<-1$, because exerting control is strictly sub-optimal (of course assuming that $V$ is smooth). On the contrary, we expect that for $y\ge b(x)$ exerting control be optimal, hence $V_x(x,y)=-1$.
Finally, it is clear by the problem formulation that if $x=y$, then $\gamma=0$, $\P$-a.s., and $V(x,x)=0$.

The informal discussion above translates into the following free boundary problem: find a pair $(V,b)$  that satisfies
\begin{align} \label{eq:dividend}
\begin{array}{lll}
(i) & \L V (x,y)= 0 &\mbox{for $x<y<b(x)$}, \\
(ii) & V(x,x) =0 &\mbox{for all $x> 0$},\\
(iii)& V_x(x,y) = -1 &\mbox{for $y \ge b(x)$, $y>0$}, \\
(iv) & V_x(x,y)\le -1&\mbox{for all {$y\ge x> 0$}},\\
(v) & \L V (x,y)\le 0& \mbox{for a.e.\ $y\ge x> 0$}.
\end{array}
\end{align}
The first equation $(i)$ corresponds to the martingale property of $t\mapsto e^{-rt}V(x,Y_t)$ when $x<Y_t<b(x)$. The second equation $(ii)$ is the absorption condition, whereas $(iii)$ and $(iv)$ identify the optimal boundary in terms of the so-called marginal cost of exerting control. Finally, condition $(v)$ relates to the super-martingale property of the value process. Common wisdom on singular control problems with dynamics similar to ours (e.g., \cite{GT,MZ}) suggests that we should additionally impose a so-called smooth-fit condition at the boundary of the form
\begin{align}\label{eq:smfit-h}
V_{xy}(x,b(x))=0,\quad x> 0.
\end{align}

First, plugging the boundary condition $V(x,x)=0$ into $(i)$ of \eqref{eq:dividend} we get
\begin{equation}\label{eq:boundaryeq}
\frac{\sigma^2(y)}{2} V_{yy}(x,x) + \mu(x) V_y(x,x)=0.
\end{equation}
Second, formally differentiating $V(x,x)=0$ twice with respect to $x$, we get
\begin{equation}\label{eq:boundaryeq2}
V_x(x,x)+V_y(x,x)=0 \quad \mbox{and} \quad	V_{xx}(x,x) + V_{yy}(x,x) + 2V_{xy}(x,x) =0.
\end{equation}
From \eqref{eq:boundaryeq} and the first equation in \eqref{eq:boundaryeq2} we get
\[
\frac{\sigma^2(x)}{2} V_{yy}(x,x)=\mu(x)V_x(x,x).
\]
Substituting in the second equation of \eqref{eq:boundaryeq2} we arrive at 
\begin{equation}\label{eq:lastcondition}
\frac{\sigma^2(x)}{2}(V_{xx}(x,x)+ 2V_{xy}(x,x))+ \mu(x) V_x(x,x)=0.
\end{equation}

At this point we notice that it is possible to reduce the free boundary problem for $(V,b)$ to a somewhat easier one for $(V_x,b)$. Such simplification leads us to the particular choice of candidate solutions of the form $v^b$ described in \eqref{eq:Uc} and to the connection with problems of optimal stopping.
This can be viewed as the extension of the original ideas in \cite{BC} to the current case of two-dimensional degenerate dynamics with absorption.
Indeed, setting $u:=-V_x$ and differentiating with respect to $x$ equation $(i)$ in \eqref{eq:dividend} we obtain a boundary value problem 
\begin{align} \label{eq:u}
\begin{array}{lll}
(i) & \L u (x,y)= 0 &\mbox{for $0<x<y<b(x)$}, \\
(ii)& u(x,y) = 1 &\mbox{for $y \ge b(x)$}, \\
(iii) & u(x,y)\ge 1 &\mbox{for $y\geq x>0$}.
\end{array}
\end{align}
The condition \eqref{eq:smfit-h} translates into the classical smooth-fit condition for $u$:
\begin{align}\label{eq:smf-u}
u_y(x,b^*(x))=0,\quad x> 0.
\end{align}
Moreover, the boundary condition \eqref{eq:lastcondition} on the diagonal translates into a reflection/creation equation
\begin{equation}
\label{ubc}
\frac{\sigma^2(x)}{2}\left(u_x(x,x)+2u_y(x,x)\right)+\mu(x)u(x,x)=0.
\end{equation}

A solution of $(i)$ in \eqref{eq:u} must be of the form
\[
u(x,y)=A(x)\psi(y)+B(x)\varphi(y),
\]
by definition of functions $\varphi$ and $\psi$ introduced in Section \ref{sec:setting}.
Continuous-fit $u(x,b(x))=1$ ($(ii)$ in \eqref{eq:u}) gives
\[
A(x)\psi(b(x))+B(x)\varphi(b(x))=1
\]
and the smooth-fit \eqref{eq:smf-u} implies 
\[
A(x)\psi'(b(x))+B(x)\varphi'(b(x))=0.
\]
Solving for $A$ and $B$ we obtain 
\begin{equation}\label{eq:constants}
A(x)=-\frac{\varphi'(b(x))}{S'(b^*(x))} \quad \mbox{and} \quad B(x)=\frac{\psi'(b(x))}{S'(b(x))}
\end{equation}
where we recall that $S'=\varphi\psi'-\psi\varphi'$. The ansatz $u=-V_x$ gives 
\begin{align}\label{eq:v-conj}
V(x,y):=\varphi(y)\int_x^y \frac{\psi'(b(z))}{S'(b(z))}\ud z-\psi(y)\int_x^y \frac{\varphi'(b(z))}{S'(b(z))} \ud z ,\quad \text{for $x<y<b(x)$},
\end{align}
by simply taking $V(x,y)=\int_x^y u(z,y)\ud z$ so that $V(x,x)=0$. Thus we have obtained exactly the expression of $\bar v^{b}$ in \eqref{eq:Uc}.

Next, we make use of \eqref{eq:lastcondition} (or equivalently of \eqref{ubc}) to determine the equation for $b$. Computing the derivatives $V_x$, $V_{xx}$ and $V_{xy}$ directly from \eqref{eq:v-conj}
and imposing \eqref{eq:lastcondition} we find that $b$ must solve \eqref{eq:ODE} (these calculations are performed below in \eqref{eq:vbar-y}, \eqref{eq:vbar-xy} and \eqref{eq:vbar-yy}). 
That completes a heuristic derivation of \eqref{eq:Uc} and \eqref{eq:ODE}.

It can be checked with tedious but straightforward algebra that if $b$ solves \eqref{eq:ODE}, then 
\begin{equation}
\label{LV}
\L V(x,y)=\int_x^y\L u(z,y)\,\ud z.
\end{equation}
If in addition to $(i)-(iii)$ in \eqref{eq:u}, also 
$\L u\leq 0$ a.e., then \eqref{LV} implies that 
$(v)$ in \eqref{eq:dividend} is fulfilled. It turns out that the condition $\L u\leq 0$ can be obtained by defining the function $u$ as the value function of a suitable optimal stopping problem for a carefully specified, two-dimensional degenerate diffusion that gives rise to the reflection/creation condition \eqref{ubc} (see Section~\ref{sec5} for details). 
\begin{remark}
Conditions $(i)$ in \eqref{eq:dividend} and $\L u\leq 0$ a.e.\ hold simultaneously only if the function $b$ solving \eqref{eq:ODE} is non-decreasing. In general \eqref{eq:ODE} could exhibit non-monotonic solutions and the set $\B$ in \eqref{eq:Bcirc} could be empty. In that case there seems to be no connection between the derivative $V_x$ of our singular control problem and the value $u$ of a stopping problem.
\end{remark}

\begin{remark}\label{rem:extensions}
It is clear at this point that an analogous heuristic procedure could be applied to problems with a more general structure of the payoff as, e.g.,
\[
J(x,y;D)=\E_{x,y}\Big[\int_0^\gamma e^{-rt}f(X^D_t,Y_t)\ud t+\int_{[0,\gamma]}e^{-rt}\ud D_t+e^{-r\gamma}g(X^D_\gamma)\Big].
\]
Some changes are required in the free boundary problem in Eq.\ \eqref{eq:dividend}. In particular, in $(i)$ and $(v)$ one has $-f(x,y)$ on the right-hand side of the expressions and in $(ii)$ one has $V(x,x)=g(x)$. Then, making the appropriate changes in the derivation above we can obtain the candidate expression for $V$ and the ODE for $b^*$. Of course, it is a difficult task to determine whether the ODE for the boundary admits a minimal solution that stays above the diagonal and, in general, this should not be expected. Nevertheless, it is an interesting question to find sufficient conditions for the applicability of the maximality principle in such more general setting. We leave it for future study.
\end{remark}

%%%%%%%%%%%%%%%%%%%%%%%%%%%%%%%%%%%%%%%%%%%%%%%

\section{Proof of Eq.\ \eqref{eq:vV} in Theorem~\ref{thm:mainthm}}\label{sec:proof}
\label{sec3}

In this section we prove the first result in Theorem \ref{thm:mainthm}: $V=v^*$ for $0<x\leq y$. We thus enforce throughout that the assumptions of the theorem are fulfilled, 
i.e., 
\[
\text{$b^*\in\B$ with $\P_{x,y}(\gamma^{*}< \infty)=1$ and}\:\: \E_{x,y}\Big[\sup_{t\geq 0}\left\{e^{-rt}v^*(x,Y_t)\right\}\Big]< \infty.
\]
One may notice that a.s.\ finiteness of $\gamma^*$ and the integrability condition for $v^*$ are not needed to prove Proposition \ref{fbp}. Instead those conditions will be used in the proof of the subsequent Proposition \ref{prop:Vaineq}.

We denote $U^\circ:=\{(x,y):0< x\leq y\}$ and recall that $v^*=v^{b^*}$ as in \eqref{eq:Uc}.

\begin{proposition}\label{fbp}
We have $v^*,v^*_{y}\in  C^1(U^\circ)$ and the function $v^*$ satisfies
\begin{align} \label{eq:thPDE}
\begin{array}{lll}
(i) & \L v^* (x,y)= 0 &\mbox{for {$x< y< b^*(x)$}}, \\
(ii) & v^*(x,x) =0 &\mbox{for all $x> 0$},\\
(iii)& v^*_x(x,y) = -1 &\mbox{for $y \ge b^*(x)$, {$y>0$}}, \\
(iv)& v^*_x(x,y)\le -1 &\mbox{for $(x,y)\in U^\circ$},\\
(v) & \L v^*(x,y)\le 0 &\mbox{for  $(x,y)\in U^\circ$},
\end{array}
\end{align}
where $\L$ acts on the second variable in (i) and (v). Moreover, 
the additional boundary conditions 
\begin{align}\label{eq:smf+ref}
v^*_{xy}(x,b^*(x))=0\quad\text{and}\quad\frac{\sigma^2(x)}{2}(v^*_{xx}(x,x)+ 2v^*_{xy}(x,x))+ \mu(x) v^*_x(x,x)=0
\end{align}  
hold for $x>0$.
\end{proposition}

\begin{proof}
Throughout the proof we use the notation $\bar v^*=\bar v^{b^*}$ as in \eqref{eq:vbbar}. Conditions $(ii)$ and $(iii)$ in \eqref{eq:thPDE} follow from \eqref{eq:Uc}.
The continuity of $v^*_x$ is immediate using $C^1$-regularity of $b^*$ and of its inverse $(b^*)^{-1}$ and recalling that $S'=\varphi\psi'-\varphi'\psi$ (notice in particular that $\bar v^*_x(b^{-1}_*(y),y)=-1$, which will be used next). For the continuity of $v_y^*$ take $y> b^*(x)$ (i.e., $x< (b^*)^{-1}(y)$), so that it follows from \eqref{eq:Uc}
\begin{align}\label{eq:vb-x}
v^*_y(x,y)=&\bar v^*_y\big((b^*)^{-1}(y),y\big)\\
&+\big[\bar v^*_x\big((b^*)^{-1}(y),y\big)+1\big]
\frac{1}{(b^*)'\circ (b^*)^{-1}(y)}
=\bar v^*_y\big((b^*)^{-1}(y),y\big),\notag
\end{align}
where the final equality uses $\bar v^*_x((b^*)^{-1}(y),y)=-1$. Since it is easy to check that $v_y^*$ is continuous separately in the sets $y> b^*(x)$ and $y< b^*(x)$, then \eqref{eq:vb-x} also guarantees continuity across the boundary $b^*$.

Next we prove that \eqref{eq:smf+ref} holds. We have   
\begin{align}\label{eq:vbar-y}
\bar v^*_x(x,y)=\psi(y)\frac{\varphi'\big(b^*(x)\big)}{S'\big(b^*(x)\big)}-\varphi(y)\frac{\psi'\big(b^*(x)\big)}{S'\big(b^*(x)\big)}
\end{align}
and 
\begin{align}\label{eq:vbar-xy}
\bar v^*_{xy}(x,y)=\psi'(y)\frac{\varphi'\big(b^*(x)\big)}{S'\big(b^*(x)\big)}-\varphi'(y)\frac{\psi'\big(b^*(x)\big)}{S'\big(b^*(x)\big)}.
\end{align}
Then, for $x\leq y<b^*(x)$ 
we have $v^*_{xy}(x,y)=\bar v^*_{xy}(x,y)$ and for $y> b^*(x)$ we have $v^*_{xy}(x,y)=0$. Hence, we conclude that $v^*_{xy}\in C(U^\circ)$ by taking limits in \eqref{eq:vbar-xy} for $(x,y)$ converging to the boundary (i.e., $y = b^*(x)$) where $v^*_{xy}(x,b^*(x))=0$.

In order to check the second condition in \eqref{eq:smf+ref} we must compute $v^*_{xx}$. Since $v^*=\bar v^*$ close to the diagonal $x=y$, differentiating \eqref{eq:vbar-y} and then setting $x=y$ we obtain
\begin{align}\label{eq:vyyb}
v^*_{xx}(x,x)
&=\frac{ (b^*)'(x)}{S'\big(b^*(x)\big)}\bigg[\psi(x)\left(\varphi''\big(b^*(x)\big)-\varphi'\big(b^*(x)\big)\frac{S''\big(b^*(x)\big)}{S'\big(b^*(x)\big)}\right)\\
&\qquad\qquad\qquad\qquad-\varphi(x)\left(\psi''\big(b^*(x)\big)-\psi'\big(b^*(x)\big)\frac{S''\big(b^*(x)\big)}{S'\big(b^*(x)\big)}\right)\bigg].\notag
\end{align}
The latter expression can be substantially simplified by using that 
\[
\frac{S''\big(b^*(x)\big)}{S'\big(b^*(x)\big)}=-\frac{2\mu\big(b^*(x)\big)}{\sigma^2\big(b^*(x)\big)}
\] 
combined with the fact that $\L\varphi=\L\psi=0$. Then we get
\begin{align}\label{eq:vbar-yy}
v^*_{xx}(x,x)=(b^*)'(x)\frac{2 r}{\sigma^2\big(b^*(x)\big) S'\big(b^*(x)\big)}\left[\psi(x)\varphi\big(b^*(x)\big)-\varphi(x)\psi\big(b^*(x)\big)\right].
\end{align}
Putting together \eqref{eq:vbar-y}, \eqref{eq:vbar-xy} and \eqref{eq:vbar-yy} and using that $(b^*)'(x)=F(x,b^*(x))$ we obtain the second equation in \eqref{eq:smf+ref}.

Since we have already proven that $v^*_y$ and $v^*_{xy}$ are continuous in $U^\circ$, it remains to show that $v^*_{yy}$ is also continuous to conclude that $v^*_y\in C^1(U^\circ)$. Since $b\in C^1\big((0,\infty)\big)$ it is easy to check that 
$\bar v^*_{yy}\in C(U^\circ)$.
In order to show that $v^*_{yy}$ is also continuous across the boundary, we differentiate \eqref{eq:vb-x} once more and use the first condition in \eqref{eq:smf+ref} to get
\begin{align}\label{eq:vyy}
v^*_{yy}(x,y)=\bar v^*_{yy}\big((b^*)^{-1}(y),y\big),\quad \text{for all $y>b^*(x)$ (i.e., $x<(b^*)^{-1}(y)$)}.
\end{align}
Since $v^*_{yy}(x,y)=\bar v^*_{yy}(x,y)$ for $x<y<b^*(x)$ we have the desired regularity across the boundary.

We next show $(i)$ in \eqref{eq:thPDE}. By direct calculations {on \eqref{eq:Uc}} and $\L\varphi=\L\psi=0$ we obtain for $x<y<b^*(x)$
\begin{align*}
\L v^*(x,y)
&=\frac{\sigma^2(y)}{2}\bigg\{(b^*)'(y)\Big[
\frac{\varphi(y)}{S'\big(b^*(y)\big)}\Big(\psi''\big(b^*(y)\big)-\frac{S''\big(b^*(y)\big)}{S'\big(b^*(y)\big)}\psi'\big(b^*(y)\big)\Big) \\
&\qquad\qquad\qquad\qquad\qquad
-\frac{\psi(y)}{S'\big(b^*(y)\big)}
\Big(\varphi''\big(b^*(y)\big)-\frac{S''\big(b^*(y)\big)}{S'\big(b^*(y)\big)}\varphi'\big(b^*(y)\big)\Big)
 \Big]\\
&\qquad\qquad\qquad+2\frac{\varphi'(y)\psi'\big(b^*(y)\big)-\psi'(y)\varphi'\big(b^*(y)\big)}{S'\big(b^*(y)\big)}\bigg\}\\
&\qquad +\mu(y)\bigg(\frac{\varphi(y)\psi'\big(b^*(y)\big)-\psi(y)\varphi'\big(b^*(y)\big)}{S'\big(b^*(y)\big)}\bigg).
\end{align*}
Comparing with the expressions on the right-hand side of \eqref{eq:vbar-y}, \eqref{eq:vbar-xy} and \eqref{eq:vyyb} we obtain
\begin{align}\label{eq:refl}
\L v^*(x,y)=-\frac{\sigma^2(y)}{2}\left(v^*_{xx}(y,y)+2 v^*_{xy}(y,y)\right)-\mu(y)v^*_x(y,y)=0,
\end{align}
where the final equality is from \eqref{eq:smf+ref}.

We now show that $v^*$ satisfies also $(iv)$ in \eqref{eq:thPDE}.
For a fixed $x\in(0,\infty)$, $v_x^*(x,\,\cdot\,)\in C^1([x,\infty))$ and it solves (in the classical sense)
\begin{align} \label{eq:vy}
\begin{array}{lll}
& \L v_x^* (x,y)= 0 &\mbox{for $x<y<b^*(x)$} \\[+3pt]
& v^*_x(x,y) = -1 &\mbox{for $y \ge b^*(x)$} \\[+3pt]
& (\partial_y v^*_{x})(x,y)=0 &\mbox{for $y\ge b^*(x)$}.
\end{array}
\end{align}
The claim is thus trivial for $y\ge b^*(x)$. Let us consider $x\leq y<b^*(x)$. Plugging the second and third equation of \eqref{eq:vy} into the first one we get
\[
(\partial_{yy}v^*_{x})\big(x,b^*(x)-\big):=\lim_{y\uparrow b^*(x)}(\partial_{yy}v^*_{x})\big(x,y\big)=-\frac{2r}{\sigma^2\big(b^*(x)\big)}<0.
\]
Thus $(\partial_y v^*_{x})(x,\,\cdot\,)>0$ on $\big(b^*(x)-\eps,b^*(x)\big)$ for some $\eps>0$ and consequently $v^*_x(x,\,\cdot\,)$ is increasing on that neighbourhood. Then $v^*_x(x,\,\cdot\,)<-1$ in $\big(b^*(x)-\eps,b^*(x)\big)$ due to the second equation in \eqref{eq:vy}. Next, we want to show that 
\begin{align}\label{eq:vxy*}
(\partial_y v^*_{x})(x,y)> 0,\quad\text{for $y\in(x,b^*(x))$},
\end{align}
so that we can conclude that 
\begin{align}\label{eq:vx-1}
v^*_x(x,y)< -1,\quad\text{for $y\in(x,b^*(x))$}.
\end{align}
By arbitrariness of $x\in(0,\infty)$, we would then have the desired inequality $v^*_x\le -1$ in $U^\circ$.

With $\eps>0$ as above let
\[
y_0(x):=\sup\{y\in(x,b^*(x)-\eps]:(\partial_y v^*_{x})(x,y)\le 0\},\]
with $\sup\varnothing=x$.
For notational simplicity we drop the dependence on $x$ in $y_0(x)=y_0$. Arguing by contradiction, assume that $y_0>x$ so that $(\partial_y v^*_{x})(x,y_0)=0$. At the same time $v^*_x(x,y_0)\le -1$, because $(\partial_y v^*_{x})(x,\,\cdot\,)>0$ on $(y_0,b^*(x))$ by construction. Plugging the latter two expressions into the first equation of \eqref{eq:vy} gives 
\[
(\partial_{yy}v^*_{x})\big(x,y_0\big)\le-\frac{2r}{\sigma^2\big(y_0\big)}<0.
\]
That implies $(\partial_y v^*_{x})(x,\,\cdot\,)<0$ on $(y_0,y_0+\eps')$ for some $\eps'>0$, which is a contradiction with the definition of $y_0$.   

Having established that $v^*(x,\,\cdot\,)\in C^2\big([x,\infty)\big)$ we can prove also that $(v)$ in \eqref{eq:thPDE} holds. For $y>b^*(x)$, using \eqref{eq:Uc}, \eqref{eq:vb-x} and \eqref{eq:vyy} we obtain
\begin{align*}
\L v^*(x,y)
&=\frac{\sigma^2(y)}{2}v^*_{yy}\big((b^*)^{-1}(y),y\big)\!+\!\mu(y)v^*_{y}\big((b^*)^{-1}(y),y\big)\!-\!r v^*\big((b^*)^{-1}(y),y\big)\!-\!r\big((b^*)^{-1}(y)\!-\!x\big)\\
&=(\L \bar v^*)((b^*)^{-1}(y),y)\!-\!r\big((b^*)^{-1}(y)v-\!x\big),
\end{align*}
where the second equality is by continuity of $v^*_y$ and $v^*_{yy}$ at the boundary $b^*$. The same continuity and $(i)$ in \eqref{eq:thPDE} allow us to conclude 
$$
\L v^*(x,y)= -r\big((b^*)^{-1}(y)-x\big)\le 0,
$$
as needed.
\end{proof}

\begin{remark}\label{rem:fbp}
Proposition \ref{fbp} does not use any specific property of $b^*$ other than the fact that $b^*\in \B$. Therefore, all the results in that proposition continue to hold for any $v^b$ associated to $b\in\B$ (with $b^*$ replaced by $b$ everywhere).
\end{remark}

\begin{proposition} \label{prop:Vaineq}
We have $v^*= V$ on $U^\circ$. 
\end{proposition}

\begin{proof}
Fix $0<x\le y$, let $D\in\A_{x,y}$ be an arbitrary control and denote
\[
\tau_n:=\inf\{t\ge 0:\langle M^{D}\rangle_t\ge n\},\qquad n\ge 1,
\]
where
\[
M^{D}_t = \int_0^{t\wedge\gamma} e^{-rs}v^*_y(X^D_s,Y_s) \sigma(Y_s) \,\ud W_s,\qquad t\ge 0,
\]
is a local martingale. Set $\gamma_n:=\gamma\wedge\tau_n$ and apply It\^o's formula to get
\begin{align*}
 e^{- r(t\wedge\gamma_n)}v^*(X^D_{t\wedge\gamma_n}, Y_{t\wedge\gamma_n}) &= v^*(x,y)
+ \int_0^{t\wedge\gamma_n} e^{-rs}\L v^*(X^D_s,Y_s) \, \ud s +M^{D}_{t\wedge\gamma_n} \\
 &\quad+\!\int_0^{t\wedge\gamma_n}\!\! e^{- rs} v^*_x(X^D_s,Y_s) \ud D^c_s + \!\!\sum_{s <t\wedge\gamma_n}\! e^{- rs} \left(v^*(X^D_s,Y_s)-v^*(X^D_{s-},Y_{s})\right) ,
\end{align*}
where $D^c$ denotes the continuous part of $D$. Now, using that $v^*_x\le -1$ and $\L v^*\le 0$ on $U^\circ$, and 
that 
\[
v^*(X^D_s,Y_s)-v^*(X^D_{s-},Y_{s})=\int_{0}^{D_s-D_{s-}}v^*_x(X^D_{s-}+z,Y_{s})\ud z\le -(D_s-D_{s-})=-\Delta D_s,
\]
we have
\begin{align}\label{eq:verif0}
e^{- r(t\wedge\gamma_n)}v^*(X^D_{t\wedge\gamma_n},Y_{t\wedge\gamma_n}) \leq v^*(x,y) + M^{D}_{t\wedge\gamma_n}- \int_{0-}^{t\wedge\gamma_n} e^{- rs} \ud D_s. 
\end{align}

Taking expectation and using that $v^*\ge 0$ we arrive at
\[
v^*(x,y)\ge \E_{x,y}\left[\int_{0}^{t\wedge\gamma_n} e^{- rs} \ud D_s\right].
\]
Now, letting $n\to\infty$ and then $t\to\infty$ we obtain
\[
v^*(x,y)\ge \E_{x,y}\left[\int_{0}^{\gamma} e^{- rs} \ud D_s\right]
\]
by monotone convergence. Since $D\in\A_{x,y}$ was arbitrary, it follows that $v^*\geq V$.

For the other inequality, let $K_n:=[0,n]^2$ and recall $X^*=X^{b^*}$ and $\gamma^*$ from Theorem \ref{thm:mainthm}. Set 
\[
\rho_n:=\inf\{t\ge 0:(X^*_t,Y_t)\notin K_n\}\quad\text{and $\gamma^*_n:=\gamma^*\wedge\rho_n$}.
\]
Apply the same arguments as above to arrive at
\begin{align*}
 e^{- r(t\wedge\gamma^*_n)}v^*(X^*_{t\wedge\gamma^*_n}, Y_{t\wedge\gamma^*_n}) &=v^*(x,y)
 + \int_0^{t\wedge\gamma^*_n} e^{-rs}\L v^*(X^*_s,Y_s) \ud s +M^{D}_{t\wedge\gamma^*_n} \\
 &\quad+ \int_0^{t\wedge\gamma^*_n} e^{- rs} v^*_x(X^*_s,Y_s) \ud D^{*,c}_s +  \int_{0}^{D^*_0}v^*_x(x+z,y)\ud z,
\end{align*}
since $t\mapsto D^{*}_t$ is continuous for $t\in(0,\infty)$.
By construction $(X^*,Y)$ is bound to evolve in the set $\{(x,y):x\le y\le b^*(y)\}$ by Lemma \ref{lem:SK}, so we have that $\L v^*(X^*_s(\omega),Y_s(\omega))=0$, for all $s\in[0,t\wedge\gamma^*_n(\omega)]$ $\P$-a.s.\ by Proposition \ref{fbp}. Taking also expectations in the expression above and re-arranging terms we arrive at
\begin{align*}
v^*(x,y)=&\E_{x,y}\Big[e^{- r(t\wedge\gamma^*_n)}v^*(X^*_{t\wedge\gamma^*_n}, Y_{t\wedge\gamma^*_n}) +\int_0^{t\wedge\gamma^*_n} e^{- rs} \ud D^{*}_s \Big],
\end{align*}
where we used the fact that $v^*_x(X^*_s,Y_s)\ud D^{*,c}_s=v^*_x(X^*_s,b^*(X^*_s))\ud D^{*,c}_s=-\ud D^{*,c}_s$ since $s\mapsto D^*_s(\omega)$ only increases when $Y_s=b^*(X^*_s)$ by \eqref{eq:SKsol}. The function $v^*(\,\cdot\,,y)$ is decreasing thanks to $(iv)$ in \eqref{eq:thPDE}. So, thanks to the integrability condition \eqref{eq:integr}, letting $t$ and $n$ go to infinity and using that $v^*(X^*_{\gamma^*},Y_{\gamma^*})=0$, $\P_{x,y}$-a.s.\ we obtain
\[
v^*(x,y)=\E_{x,y}\Big[\int_0^{\gamma^*} e^{- rs} \ud D^{*}_s \Big]\le V(x,y).
\]
Since we already have the reverse inequality $v^*\ge V$, the claim $v^*=V$ and the optimality of 
$D^*$ follow. 
\end{proof}

\begin{remark}
By Proposition~\ref{prop:Vaineq}, the equality $V= v^*$ holds on $U^\circ$. If $v^*$ can be extended 
so that the properties specified in Proposition \ref{fbp} hold on $U=\{(x,y):0\leq x\leq y\}$, then the equality extends to $U$. This will be the case in our example with gBm in Section \ref{sec8}.
\end{remark}

\section{The maximality principle}\label{sec:max}
\label{sec4}
In this section we prove \eqref{inf} in Theorem~\ref{thm:mainthm}. In particular, solutions of \eqref{eq:ODE} are ordered, and we show that a large $b\in \B$ gives a large $v^b$.
This leads to the characterisation of $b^*$ in Theorem~\ref{thm:mainthm} as the minimal element of $\B$ and, consequently, the only one for which \eqref{eq:integr} holds (notice that $(b^*)^{-1}$ is the maximal inverse of an element of $\B$).

First note that we have existence and uniqueness up to a possible explosion of the solution to \eqref{eq:ODE} for any initial point in the interior of $U$. Indeed, since $\mu$ and $\sigma$ are locally Lipschitz continuous on $(0,\infty)$, one finds that
\begin{itemize}
\item
the map $(x,y)\mapsto F(x,y)$ is continuous for all $0\le x<y$;
\item
$F(x,\cdot\,)$ is locally Lipschitz on $(x,\infty)$.
\end{itemize}
Thus a unique solution of \eqref{eq:ODE} can be constructed by standard Picard-Lindel\"of type of arguments. Clearly this does not guarantee $\B\neq\varnothing$, for which we will show sufficient conditions in Proposition \ref{prop:monotonicity}.
Since $F$ is not defined on the diagonal, due to
$\delta(y,y)=0$, then solutions of \eqref{eq:ODE} may approach the diagonal with an infinite slope.
More precisely, let $\eta>0$, and consider an initial point $(\xi,\eta)$ with $0\le \xi<\eta$.
Then, for any $\ep\in(0,\eta-\xi)$, there exists a unique $C^1$ solution of \eqref{eq:ODE}, with $b(\xi)=\eta$, on the interval $(x_\ep,x^\eps)$, where
\begin{align}\label{eq:ye}
x_\ep:=\sup\{x\in(0,\xi):b(x)\le x+\ep\} \vee 0\quad {\text{and}}\quad
x^\eps:=\inf\{x>\xi: b(x)\le x+\eps\}.
\end{align}
Moreover, the solution can be extended continuously to $(x_0,x^0]$, with $b(x^0)=x^0$.
Furthermore, for $\eta_1\le\eta_2$, two solutions $b_1,b_2\in\B$ with initial points $(\xi,\eta_1)$ and $(\xi,\eta_2)$, respectively, satisfy $b_1\leq b_2$.

In order to properly define the minimal element of $\B$, for $\eta>\xi> 0$ let us denote by $b^{\xi,\eta}$ a solution of \eqref{eq:ODE} such that $b(\xi)=\eta$. It is convenient to introduce the set of initial values $0< \xi<\eta$ for which $b^{\xi,\eta}\in\B$. That is, for a fixed $\xi> 0$ we let
\begin{align}\label{eq:Xi}
\mathcal E_\xi:=\{\eta>\xi: b^{\xi,\eta}\in\B\}.
\end{align}
Then, for any $\eta_1<\eta_2$ in $\mathcal E_\xi$ we must have $b^{\xi,\eta_1}<b^{\xi,\eta_2}$ by uniqueness of the solution. We will say that solutions are ordered and we refer to `larger' or `smaller' solution as appropriate.

%%%%%%%%%%%%%%%%%%%%%%%%%%%%%%%%%%%%%%%%%%%%%%

\begin{proposition}\label{prop:increasing}
Fix $\xi> 0$ and assume that $\eta_1>\eta_2$ belong to $\mathcal E_\xi$ so that $b_1:=b^{\xi,\eta_1}>b^{\xi,\eta_2}=:b_2$ satisfy $b_1,b_2\in\B$. Then, $v^{b_1}> v^{b_2}$ on $U^\circ$.
\end{proposition}

\begin{proof} 
Let us start by recalling that all the results in Proposition \ref{fbp} also hold for $v^{b_1}$ and $v^{b_2}$ provided that we replace $b^*$ therein by $b_1$ and $b_2$, respectively (see Remark \ref{rem:fbp}). In particular, $v^{b_i}\in C^1(U^\circ)$ and $v_x^{b_i}\leq -1$ by \eqref{eq:thPDE}. Moreover, arguing as in the proof of $(iv)$ of Eq.\ \eqref{eq:thPDE} we can analogously show that 
\begin{align}\label{eq:comp0}
\partial_y v^{b_i}_x(x,y) >0,\qquad\text{ for $x<y<b_i(x)$ and $i=1,2$}.
\end{align}
That implies, in particular, 
\[
v^{b_i}_x(x,y)<-1,\qquad\text{for $b^{-1}_i(y)<x<y$ and $i=1,2$}, 
\]
which will be used in the next part of the proof.

Next we prove that $v^{b_1}> v^{b_2}$. To simplify notation we set $v^i=v^{b_i}$, for $i=1,2$. 
It is sufficient to prove 
\begin{align}\label{eq:vxineq}
v^1_x\le v^2_x
\end{align} 
with {\em strict inequality at the diagonal}, so that for each $y>0$ and all $0\le x<y$ we have
\[
v^1(x,y)=-\int_x^y v^1_x(z,y)\ud z>-\int_y^x v^2_x(z,y)\ud z=v^2(x,y),
\] 
where we also use that $v^1(y,y)=v^2(y,y)=0$ by \eqref{eq:Uc}. As in \eqref{eq:vy}, for a fixed $x\in(0,\infty)$ and for $i=1,2$, we have by construction that $v_x^{i}(x,\,\cdot\,)\in C^1([x,\infty))$ and it solves
\begin{align} \label{eq:vy2}
\begin{array}{lll}
& \L v_x^{i} (x,y)= 0 &\mbox{for $x<y<b_i(x)$} \\[+3pt]
& v^{i}_x(x,y) = -1 &\mbox{for $y \ge b_i(x)$} \\[+3pt]
& (\partial_y v^{i}_{x})(x,y)=0 &\mbox{for all $y\ge b_i(x)$}.
\end{array}
\end{align}

Since $b_1> b_2$, then we have $b^{-1}_1(y)=b^{-1}_2(y)=0$ for $0<y\le b_2(0)$ and $b_1^{-1}(y)< b_2^{-1}(y)$ for $y>b_2(0)$ (we adopt the convention that $(0,b_2(0)]=\varnothing$ if $b_2(0)=0$). Now, for $0< x\le b^{-1}_1(y)$ we have $v^1_x(x,y)=-1=v^2_x(x,y)$, so \eqref{eq:vxineq} holds with equality in that set. Instead, we have 
\begin{align}\label{eq:comp1}
v^1_x(x,y)< -1=v^2_x(x,y),\qquad\text{for $b_1^{-1}(y)<x\le b^{-1}_2(y)$},
\end{align}
and \eqref{eq:vxineq} holds with strict inequality.

It remains to prove that such strict inequality also holds on $b^{-1}_2(y)< x\le y$. This is equivalent to proving it on $x < y\le b_2(x)$ for each $x> 0$ given and fixed.
Set $u(x,y):=(v^1_x-v^2_x)(x,y)$. Then $u(x,b_2(x))< 0$ by \eqref{eq:comp1} evaluated at the boundary $y=b_2(x)$. Moreover $(\partial_y u)(x,b_2(x))>0$ by the third equation in \eqref{eq:vy2} applied to $v^2_x$, and \eqref{eq:comp0} applied to $v^1_x$ (along with the fact that $b_1>b_2$). Thus, there is $\eps>0$ such that $(\partial_y u)(x,\,\cdot\,)>0$ on $(b_2(x)-\eps,b_2(x))$ and, setting 
\[
y_0:=\sup\{y\in(x,b_2(x)-\eps): (\partial_y u)(x,y)\le 0\},
\]
we must have $y_0=x$. Indeed $\L u(x,y)=0$ for $y\in (x,b_2(x))$. In particular, at $y_0$ it holds $\partial_{yy}u(x,y_0)=2r\sigma^{-2}(y_0)u(x,y_0)<0$, where the strict inequality is by $u(x,y_0)\le u(x,b_2(x))<0$ which holds because $\partial_y u(x,\,\cdot\,)>0$ on $(y_0,b_2(x))$. This leads to a contradiction, since $y_0$ should be a minimum of $u(x,\,\cdot\,)$. Hence, $u(x,\,\cdot\,)< 0$ on $(x,b_2(x))$, as needed.  
\end{proof}

\begin{corollary}[{\bf Uniqueness and minimality}]\label{rem:unique}
There can be at most one $b^* \in \B$ that satisfies the integrability conditions \eqref{eq:integr} in Theorem~\ref{thm:mainthm}. Moreover, it must be the minimal element in $\B$ and \eqref{inf} holds.
\end{corollary}
\begin{proof}
Assume by way of contradiction that there are $b$ and $\tilde b$ in $\B$ that satisfy the integrability conditions \eqref{eq:integr}. Then, with no loss of generality $b<\tilde b$ but $V=v^b=v^{\tilde b}$, which contradicts Proposition \ref{prop:increasing}. As for the minimality of $b^*$, assume that there is $b\in\B$ with $b<b^*$ and $b^*$ satisfies the integrability conditions in \eqref{eq:integr}. Then, by Proposition \ref{prop:increasing}, $v^b<v^{b^*}$ so that $v^b$ satisfies the second condition in \eqref{eq:integr}. Moreover, by construction $D^*_t\le D^b_t$ for all $t\ge 0$, $\P_{x,y}$-a.s. Hence, $\gamma^b\le \gamma^*$, $\P_{x,y}$-a.s.\ and also the first condition in \eqref{eq:integr} holds for $b$. Thus, we reach again a contradiction as $b\in\B$ would satisfy \eqref{eq:integr} and it would be $v^b=V=v^*$.
\end{proof}

We now address the question of whether $\B$ is non-empty. We first have the following lemma.

\begin{lemma}\label{lem}
For any sequence $(x_n,y_n,)_{n\ge 1}$ with $y_n>x_n$ such that $\lim_{n\to\infty}y_n=\lim_{n\to\infty}x_n=p>0$ as $n\to\infty$ we have $F(x_n,y_n)\to-\infty$ as $n\to\infty$.
\end{lemma}

\begin{proof} \label{le:2}
First observe that as $n\to\infty$, letting $\eps_n:=y_n-x_n$ we have the asymptotic expansion
\begin{align*}
\delta(x_n,y_n)=\psi(x_n)\varphi(x_n)\left(\frac{\varphi'(x_n)}{\varphi(x_n)}-\frac{\psi'(x_n)}{\psi(x_n)}\right)\eps_n+o(\eps_n).
\end{align*}
Likewise, for the first term in the square brackets in the definition of $F$ we have (see \eqref{eq:F})
\begin{align*}
\varphi'(x_n)\psi'(x_n)\left(\frac{\varphi'(y_n)}{\varphi'(x_n)}-\frac{\psi'(y_n)}{\psi'(x_n)}\right)
&=\varphi'(x_n)\psi'(x_n)\left(\frac{\varphi''(x_n)}{\varphi'(x_n)}-\frac{\psi''(x_n)}{\psi'(x_n)}\right)\eps_n+o(\eps_n)\\
&=\varphi'(x_n)\psi'(x_n)\frac{2r}{\sigma^2(x_n)}\left(\frac{\varphi(x_n)}{\varphi'(x_n)}-\frac{\psi(x_n)}{\psi'(x_n)}\right)\eps_n+o(\eps_n),
\end{align*}
where the final expression follows from the fact that both $\psi$ and $\varphi$ solve the ODE $\L f=0$.
So for the first term in \eqref{eq:F} we have
\begin{align}\label{eq:asymp}
\lim_{n\to\infty}\frac{\sigma^2(y_n)}{r\delta(x_n,y_n)}\Big(\varphi'(y_n)\psi'(x_n)-\varphi'(x_n)\psi'(y_n)\Big)=-2.
\end{align}

For the second term inside the square brackets in \eqref{eq:F} we have
\[
\lim_{n\to\infty}\Big(\varphi'(y_n) \psi(x_n)-\varphi(x_n)\psi'(y_n)\Big)= \Big(\varphi'(p) \psi(p)-\varphi(p)\psi'(p)\Big)<0,
\]
so the desired result follows easily since $\delta(x_n,y_n)\downarrow 0$ as $n\to\infty$ and $\sigma^2(p)>0$. 
\end{proof}

For $x>0$ let 
\[d(x)=\inf\{y>x:F(x,y)>0\},\]
and note that $d(x)>x$ by Lemma~\ref{lem}. 
For $\xi>0$, denote by $b^{\xi,d(\xi)}$ the solution of \eqref{eq:ODE} with initial point $b^{\xi,d(\xi)}(\xi)=d(\xi)$. Let us also recall that $x^0=x^0(\xi,d(\xi))$ is the smallest $x>\xi$ for which $b^{\xi,d(\xi)}$ touches the diagonal (see \eqref{eq:ye}). We can now provide easy sufficient conditions under which $\B\neq\varnothing$, its minimal element $b^*$ exists and the map $x\mapsto d(x)$ is strictly increasing.

\begin{proposition}\label{prop:monotonicity}
Assume that $\mu,\sigma\in C^1([0,\infty))$, $d(x)<\infty$ for all $x> 0$, and 
\begin{align}\label{eq:zeta}
\zeta(x):=\frac{2r}{\sigma^2(x)}+\left(\frac{\mu}{\sigma^2}\right)'(x)+\frac{\mu^2(x)}{\sigma^4(x)} >0,\qquad x> 0.
\end{align}
Then, $x\mapsto d(x)$ is strictly increasing and $x\mapsto b^{\xi,d(\xi)}(x)$ is increasing on $(0,\xi)$ and decreasing on $(\xi,x^0)$. Moreover,  
\[
\bar b(x)=\sup_{\xi>0}b^{\xi,d(\xi)}(x)
\] 
is the minimal element in $\B$. 
\end{proposition}

\begin{proof}
Let us fix $\xi>0$ and simplify the notation for this initial part of the proof by setting $b(x)=b^{\xi,d(\xi)}(x)$. First we show that $\xi$ is the unique stationary point of $b$ on $(0,x^0]$ and that it is a global maximum.

Assume there exists $\nu> 0$ such that $F(\nu,b(\nu))=0$ so that $b'(\nu)=0$ (a priori $\nu\neq \xi$ and potentially $b(\nu)\neq d(\nu)$). Then $b(\nu)>\nu$ by  Lemma \ref{lem} and, by \eqref{eq:F}, we have 
\begin{align}\label{eq:F=0}
\varphi'\big(b(\nu)\big)\psi'(\nu)-\psi'\big(b(\nu)\big)\varphi'(\nu)=-\frac{\mu(\nu)}{\sigma^2(\nu)}\left(\varphi'\big(b(\nu)\big)\psi(\nu)-\psi'\big(b(\nu)\big)\varphi(\nu)\right).
\end{align}
Differentiating \eqref{eq:ODE} we obtain
\[
b''(x)=F_x\big(x,b(x)\big)+F_y\big(x,b(x)\big)b'(x)
\]
so that
\begin{align}\label{eq:b''}
b''(\nu)=F_x\big(\nu,b(\nu)\big).
\end{align}
Recalling the expression for $F$ in \eqref{eq:F} and performing straightforward calculations we obtain
\begin{align*}
F_x(x,y)
&=-\frac{\delta_x(x,y)}{\delta(x,y)}F(x,y)\\
&\quad+\frac{\sigma^2(y)}{r\delta(x,y)}\bigg[\varphi'(y)\psi''(x)-\varphi''(x)\psi'(y)+(\mu/\sigma^2)'(x)\big(\varphi'(y)\psi(x)-\varphi(x)\psi'(y)\big)\\
&\qquad\qquad\qquad+(\mu/\sigma^2)(x)\big(\varphi'(y)\psi'(x)-\varphi'(x)\psi'(y)\big)\bigg].
\end{align*}
Evaluating the above at $(x,y)=(\nu,b(\nu))$ we see that the first term vanishes. We can evaluate the second term substituting $\psi''$ and $\varphi''$ therein with
\[
\psi''=(2r/\sigma^2)\psi-(2\mu/\sigma^2)\psi'\quad\text{and}\quad\varphi''=(2r/\sigma^2)\varphi-(2\mu/\sigma^2)\varphi',
\]
which are due to $\L \psi=\L \varphi =0$.
Rearranging terms and using also \eqref{eq:F=0} we thus obtain
\begin{align*}
b''(\nu)=\frac{\sigma^2\big(b(\nu)\big)}{r\delta\big(\nu,b(\nu)\big)}\zeta(\nu)\left[\varphi'\big(b(\nu)\big)\psi(\nu)-\psi'\big(b(\nu)\big)\varphi(\nu)\right],
\end{align*}
with $\zeta$ defined as in \eqref{eq:zeta}. Since $\varphi'<0$ and $\psi'>0$ on $(0,\infty)$ and $\delta>0$ (see \eqref{eq:gammaneg}), we can conclude $b''(\nu) <0$.

This means that any stationary point of $b$ on $(0,x^0]$ must be a maximum. Hence, there can only be one stationary point on $(0,x^0]$ and therefore it must coincide with $\xi$, where $b(\xi)=d(\xi)$. So, $b$ is {\em strictly} increasing on $[0,\xi)$ and {\em strictly} decreasing on $(\xi,x^0]$ (with $F(\nu,b(\nu))\neq 0$ in both intervals). Thus $b(\nu)>d(\nu)$ for $\nu\in[0,\xi)$, because if it were $b(\nu)< d(\nu)$ for some $\nu\in[0,\xi)$ then it would be $b'(\,\cdot\,)<0$ in a right-neighbourhood of $\nu$, which contradicts that $b$ is strictly increasing there. Analogously, it must be $b(\nu)<d(\nu)$ for $\nu\in(\xi,x^0]$.

Notice also that $F_y(x,d(x))\ge 0$ by definition of $d(x)$ and, by the same calculations as above, $F_x(x,d(x))<0$. Since $F(x,d(x))=0$, it follows that $x\mapsto d(x)$ is strictly increasing.
Moreover, noticing that $x<d(x)<\infty$ for all $x> 0$ the limit $d(\infty):=\lim_{x\to\infty}d(x)$ exists and it is infinite.

Clearly $\xi\mapsto b^{\xi,d(\xi)}(x)$ is also increasing (for each $x>0$ and any $\xi>x$) by uniqueness of the solution to \eqref{eq:ODE} and the construction above. Hence 
\[
\bar b(x):=\sup_{\xi>0}b^{\xi,d(\xi)}(x)
\]
is well-defined, and it satisfies the ODE \eqref{eq:ODE}. By construction $\bar b\in\B$ since $\bar b(x)>d(x)$ for all $x> 0$, which implies $F(x,\bar b(x))>0$ for all $x> 0$. Moreover, it is also the minimal element of $\B$ because for any $b\in\B$ it must be $b(x)\ge b^{\xi,d(\xi)}(x)$ for all $x> 0$ and any $\xi> 0$. 
\end{proof}

Let $\bar b$ as in Proposition \ref{prop:monotonicity} and $X^{\bar b}$ and $\gamma^{\bar b}$ as in Lemma \ref{lem:SK}. Take $v^{\bar b}$ as in \eqref{eq:Uc} and recall Corollary \ref{rem:unique}. Then the next result follows.

\begin{corollary}\label{cor:bb}
If the integrability conditions in \eqref{eq:integr} hold for $\gamma^{\bar b}$ and $v^{\bar b}$ then $\bar b=b^*$ and \eqref{inf} holds along with the rest of Theorem \ref{thm:mainthm}. 
\end{corollary}

\begin{remark}
Proposition \ref{prop:monotonicity} above gives conditions under which $\B$ is non-empty and allows to construct its minimal element $\bar b$. Corollary \ref{cor:bb} then says that if such minimal element satisfies the integrability conditions as in Theorem~\ref{thm:mainthm} then $\bar b=b^*$ and it is the optimal boundary. While we do not know of general conditions under which the integrability conditions for 
$\overline b$ are satisfied, the following observation can be useful in certain situations.

Assume that two sets of model specifications $(\mu_1,\sigma_1)$ and $(\mu_2,\sigma_2)$ are given, and assume that 
$\mu_1(y)=\mu_2(y)$ and $\sigma_1(y)=\sigma_2(y)$ for $y\geq y_0$. Further assume that $b^*_{1}\in\B_1$ satisfies the integrability conditions for the first set of parameters so that reflection along $b^*_{1}$ is optimal. Construct $b_2$ by 
setting $b_2=b^*_{1}$ for $x\geq y_0$, and by solving \eqref{eq:ODE} for $x\leq y_0$ with boundary condition $b_2(y_0)=b^*_{1}(y_0)$. Then $b_2$ is optimal for $(\mu_2,\sigma_2)$. 
\end{remark}

%%%%%%%%%%%%%%%%%%%%%%%%%%%%%%%%%%%%%%%%%%%%%%

\section{An optimal stopping problem with oblique reflection}
\label{sec5}

The link between Peskir's maximality principle and our singular control problem passes through a special connection between problems of singular stochastic control with absorption and optimal stopping problems for reflecting diffusions with discounting at the `rate' of local time (see, e.g., \cite{DeA20} and \cite{DeAE17}).

Letting $u^*:=-v^*_x$, with $v^*$ as in Theorem \ref{thm:mainthm}, and using the explicit expression \eqref{eq:Uc} we obtain 
\begin{align} \label{eq:ode-u}
\begin{array}{lll}
(i) & \L u^* (x,y)= 0 &\mbox{for $0<x<y<b^*(x)$}, \\
(ii)& u^*(x,y) = 1 &\mbox{for $y \ge b^*(x)$}, \\
(iii) & u^*(x,y)\ge 1&\mbox{for all $(x,y)\in U^\circ$},\\
(iv) &\L u^*(x,y)\leq 0 &\mbox{for a.e.\ $(x,y)\in U^\circ$}.
\end{array}
\end{align}
Notice that all but the final equation above can also be obtained by simply differentiating \eqref{eq:thPDE} with respect to $x$.
Moreover, the boundary condition in \eqref{eq:smf+ref} reads
\begin{align}\label{eq:refl-u*}
\frac{\sigma^2(x)}{2}\left(u^*_x(x,x)+2u^*_y(x,x)\right)+\mu(x)u^*(x,x)=0,\quad\text{for $x>0$}.
\end{align}
Since $u^*(x,y)>1$ for $x<y<b^*(x)$ by definition of $u^*$ and \eqref{eq:vx-1}, we see that $u^*$ appears 
to be the value function of an optimal stopping problem whose underlying process $(\hat X,\hat Y)$ has a peculiar behaviour along the diagonal due to \eqref{eq:refl-u*}. 

First we construct the process $(\hat X,\hat Y)$, then we state the optimal stopping problem and finally we make a precise claim on the connection between $V=v^*$ and $u^*$. 
For $y\geq x>0$, let $(\hat X,\hat Y)$ be defined by the system 
\begin{equation} \label{eq:reflectedSDE}
\left\{\begin{array}{ll}
\ud\hat Y_t = \mu (\hat Y_t )\ud t + \sigma(\hat Y_t) \ud W_t + \ud \hat A_t, & \hat Y_0=y, \\[+4pt]
\hat X_t = x + \frac{1}{2} \hat A_t, & \\[+4pt]
\hat Y_t \geq \hat X_t\:\:\text{and}\:\: \int_0^t 1_{\{\hat Y_t > \hat X_t \}} \ud\hat A_t=0, & \text{for all $t\ge 0$},\\
\end{array}
\right.
\end{equation}
where $(\hat A_t)_{t\ge 0}$ is a continuous and increasing process and all equations hold $\P$-a.s. The pair $(\hat X,\hat Y)$ is the solution of a two-dimensional degenerate reflecting SDE, with reflection occurring at the diagonal $\hat X=\hat Y$ but in the direction ${\bf v}=(\tfrac{1}{2},1)$. The next lemma states that such a process is uniquely determined. We postpone its proof to the appendix so that we instead can continue with the construction of our optimal stopping problem.

\begin{lemma}\label{le:existencereflectedSDE}
There exists a pathwise unique $(\F_t)_{t\ge 0}$-adapted solution $(\hat X,\hat Y)$ to \eqref{eq:reflectedSDE}. Moreover, the process $(\hat X,\hat Y)$ is strong Markov.
\end{lemma}
Let us define the process 
\[
A_t := \int_0^t \frac{\mu(\hat Y_s)}{\sigma^2(\hat Y_s)} \ud \hat A_s
\]
and the optimal stopping problem with value
\begin{equation} \label{osp}
\widehat V(x,y):= \sup_{\tau} \E_{x,y} \left [  \exp (A_\tau -r\tau)\right],\quad\text{for $0<x\le y$,}
\end{equation}
where the supremum is taken over $\P_{x,y}$-a.s.\ finite $(\F_t)$-stopping times. 
In essence, \eqref{osp} is an optimal stopping problem for the reflected strong Markov
process $(\hat X,\hat Y)$, which pays $1$ for immediate stopping and which gains value when the process $\hat Y$ is reflected in the diagonal $x=y$. 
The next result confirms that $\widehat V =u^*=-v^*_x=-V_x$.

\begin{theorem} \label{thm:stoppingconnection}
Let the assumptions of Theorem \ref{thm:mainthm} hold here (so that $v^*=V$ on $U^\circ$ and $b^*\in\B$ is optimal). Set
\[
\tau^*:=\inf\{t\ge 0: \hat Y_t\ge b^*(\hat X_t)\},
\]
recall $u^*:=-v^*_x$ and fix arbitrary $y\geq x> 0$. 
Assume $\P_{x,y}(\tau^*<\infty)=1$ and for any sequence of stopping times $(\tau_n)_{n\ge 1}$ with $\tau_n\uparrow \infty$, $\P_{x,y}$-a.s.\ as $n\to\infty$, the  transversality condition 
\begin{align}\label{trans}
\lim_{n\to\infty}\E_{x,y}\left[e^{A_{\tau_n}-r\tau_n}u^*(\hat X_{\tau_n},\hat Y_{\tau_n})1_{\{\tau_n<\tau^*\}}\right]=0
\end{align}
holds.
Then $u^*(x,y)=\widehat V(x,y)$ with $\widehat V$ as in \eqref{osp} and $\tau^*$ is an optimal stopping time.
\end{theorem}

\begin{proof}
The proof consists of a standard verification argument. Set $\C:=\{(x,y)\in U^\circ: y<b^*(x)\}$ and denote $\overline\C=\{(x,y)\in U^\circ: y\le b^*(x)\}$ and $\S=\{(x,y)\in U^\circ :y\ge b^*(x)\}$.

Thanks to the regularity of $v^*$ in Proposition~\ref{fbp}
we know that $u^*\in C^1(U^\circ)$. Then, from $(i)$ and $(ii)$ in \eqref{eq:ode-u} we also obtain that $u^*_{yy}$ is locally bounded on $U^\circ$ and it is continuous separately in the sets $\overline \C$ and $\S$, because
\[
u^*_{yy}(x,y)=2 \sigma^{-2}(-\mu u^*_y+r u^*)(x,y),\quad\text{for $x<y<b^*(x)$},
\]
and $u^*_{yy}=0$ on $\S$. 
(Notice that $u^*_{yy}$ is not continuous across the boundary.) 

For any initial points $0<x\le y$ we can apply the change-of-variable formula derived in \cite{AJ} to $e^{A_t-rt}u^*(\hat X_t,\hat Y_t)$. In particular, for any stopping time $\tau$ we obtain
\begin{align}\label{eq:Ito}
 &e^{A_{t\wedge\tau\wedge\theta_n} - r(t\wedge\tau\wedge\theta_n)}u^*(\hat X_{t\wedge\tau\wedge\theta_n}, \hat Y_{t\wedge\tau\wedge\theta_n})\\
 &= u^*(x,y) + \int_0^{t\wedge\tau\wedge\theta_n} e^{A_s - rs} \L u^*(\hat X_s, \hat Y_s)\,\ud s + M_{t\wedge\tau\wedge\theta_n} \notag \\
 &\quad+ \int_0^{t\wedge\tau\wedge\theta_n} e^{A_s -rs} \bigg[  \frac{\mu(\hat Y_s)}{\sigma^2(\hat Y_s)} u^*(\hat X_s, \hat Y_s) +u^*_y ( \hat X_s , \hat Y_s) + \frac{1}{2}u^*_x(\hat X_s, \hat Y_s) \bigg ]\ud \hat A_s\notag
\end{align}
where $M_t$ is the local martingale 
\[
M_t:=\int_0^t e^{A_s-rs}u^*_y(\hat X_s,\hat Y_s)\sigma(\hat Y_s)\ud W_s
\]
and 
\[
\theta_n:=\inf\{t\ge 0:\langle M\rangle_t\ge n\}
\]
is the usual localising sequence. Now we recall the minimality condition in the final line of \eqref{eq:reflectedSDE} that guarantees 
\[
\ud \hat A_s=1_{\{\hat X_s=\hat Y_s\}}\ud \hat A_s.
\] 
Plugging this into \eqref{eq:Ito} and recalling also \eqref{eq:refl-u*}, we see that the integral in $\ud\hat A_s$ vanishes. Then, taking expectations and rearranging terms we have
\begin{align*}
u^*(x,y)
& =\E_{x,y}\left[e^{A_{t\wedge\tau\wedge\theta_n} - r(t\wedge\tau\wedge\theta_n)}u^*(\hat X_{t\wedge\tau\wedge\theta_n},\hat Y_{t\wedge\tau\wedge\theta_n})-\int_0^{t\wedge\tau\wedge\theta_n} e^{A_s - rs} \L u^*(\hat X_s, \hat Y_s)\ud s\right]
\\
& \ge \E_{x,y}\left[e^{A_{t\wedge\tau\wedge\theta_n} - r(t\wedge\tau\wedge\theta_n)}u^*(\hat X_{t\wedge\tau\wedge\theta_n}, \hat Y_{t\wedge\tau\wedge\theta_n})\right]\\
&\ge \E_{x,y}\left[e^{A_{t\wedge\tau\wedge\theta_n} - r(t\wedge\tau\wedge\theta_n)}\right],
\end{align*}
where the first inequality comes from $(iv)$ in \eqref{eq:ode-u} (notice that $\hat X$ is constant off the diagonal and $\hat Y$ admits a transition density with respect to the Lebesgue measure) and the second one from $(iii)$. Using Fatou's lemma, we can take limits inside the expectation, as $n\to\infty$ and $t\to\infty$. Thus, by arbitrariness of the stopping time $\tau$ we have
\begin{align}\label{eq:u*U}
u^*(x,y)\ge \sup_\tau\E_{x,y}\left[{\exp(A_{\tau} - r\tau)}\right]=\widehat V(x,y).
\end{align}

For the reverse inequality, we pick $\tau=\tau^*$ in \eqref{eq:Ito} and $t=t_n$, for some deterministic $(t_n)_{n\ge 1}$ increasing to infinity. Then by $(i)$ in \eqref{eq:ode-u} we get
\begin{align*}
u^*(x,y)
&=\E_{x,y}\bigg[e^{A_{t_n\wedge\tau^*\wedge\theta_n} - r(t_n\wedge\tau^*\wedge\theta_n)}u^*(\hat X_{t_n\wedge\tau^*\wedge\theta_n}, \hat Y_{t_n\wedge\tau^*\wedge\sigma_n})-\int_0^{t_n\wedge\tau^*\wedge\theta_n} e^{A_s - rs} \L u^*(\hat X_s, \hat Y_s)\ud s\bigg]
\\
&=\E_{x,y}\left[e^{A_{t_n\wedge\tau^*\wedge\theta_n} - r(t_n\wedge\tau^*\wedge\theta_n)}u^*(\hat X_{t_n\wedge\tau^*\wedge\theta_n}, \hat Y_{t_n\wedge\tau^*\wedge\theta_n})\right]\\
&=\E_{x,y}\left[e^{A_{\tau^*} - r \tau^*}1_{\{\tau^*\le \theta_n\wedge t_n\}}+e^{A_{t_n\wedge\theta_n} - r(t_n\wedge\theta_n)}u^*(\hat X_{t_n\wedge\theta_n}, \hat Y_{t_n\wedge\theta_n})1_{\{\tau^*> \theta_n\wedge t_n\}}\right].
\end{align*}
Now, letting $n\to\infty$ and using the transversality condition \eqref{trans} we obtain
\[
u^*(x,y)=\left[e^{A_{\tau^*} - r\tau^*}\right]\le \widehat V(x,y),
\]
hence concluding the proof. 
\end{proof}

\section{The Case of Geometric Brownian motion}\label{sec8}

In this section we consider the case when $Y$ is a geometric Brownian motion, i.e.\ we assume that
$\mu(y)=\alpha y$ and $\sigma(y)=\beta y$ where $\alpha$ and $\beta$ are positive constants so that
\[
dY_t=\alpha Y_t\,\ud t + \beta Y_t\,\ud W_t.
\] 
To ensure the integrability condition \eqref{eq:integr}, we further assume that $\alpha<r$. Notice that condition \eqref{eq:zeta} in Proposition \ref{prop:monotonicity} holds.

The fundamental solutions of the equation $\L f=0$ are given by $\varphi(y)=y^{\gamma_1}$ and $\psi(y)=y^{\gamma_2}$, where $\gamma_1<0<\gamma_2<1$ are solutions of the quadratic equation
\[
\gamma^2 + (\frac{2\alpha}{\beta^2}-1)\gamma-\frac{2r}{\beta^2}=0.
\]
The function $F$ is given by
\[F(x,y)=\frac{2y}{x}-\frac{\alpha y}{rx}\frac{\gamma_2\varphi(x)\psi(y)-\gamma_1\varphi(y)\psi(x)}{\varphi(x)\psi(y)-\varphi(y)\psi(x)},\]
so $F(x,y)$ is constant along rays $y=Ax$, $A>1$. Thus $xF_x+yF_y=0$, and the function $G(z):=F(x,zx)$ is independent of the choice of $x>0$.

\begin{lemma}\label{prop}
The function $G:(1,\infty)\to\R$ is strictly increasing, with $G(1+)=-\infty$.
Moreover, $G(A)=0$ for 
\[A=\left(\frac{2r-\alpha\gamma_1}{2r-\alpha\gamma_2}\right)^{\frac{1}{\gamma_2-\gamma_1}} >1,\]
and $G(C)=C$ for 
\[C = \left(\frac{r-\alpha\gamma_1}{r-\alpha\gamma_2}\right)^{\frac{1}{\gamma_2-\gamma_1}}>A.\]
Furthermore, $G(z)-z$ is negative on $(1,C)$ and positive on $(C,\infty)$.
\end{lemma}

\begin{proof}
The claims that $G(1+)=-\infty$, $G(A)=0$ and $G(C)=C$ are straightforward to verify. For the last claim, note that the function 
\[H(z):=\frac{(G(z)-z)(\psi(z)-\varphi(z))}{z}
= (1-\frac{\alpha\gamma_2}{r}) \psi(z)-(1-\frac{\alpha\gamma_1}{r})\varphi(z)\]
is strictly increasing since $\gamma_2<r/\alpha$ and satisfies $H(C)=0$. Consequently, $H$ is negative 
on $(1,C)$ and positive on $(C,\infty)$, and so is also $G(z)-z$ because $\psi(z)-\varphi(z)>0$ for all $z>1$.
\end{proof}

\begin{figure}
\includegraphics[scale=0.2]{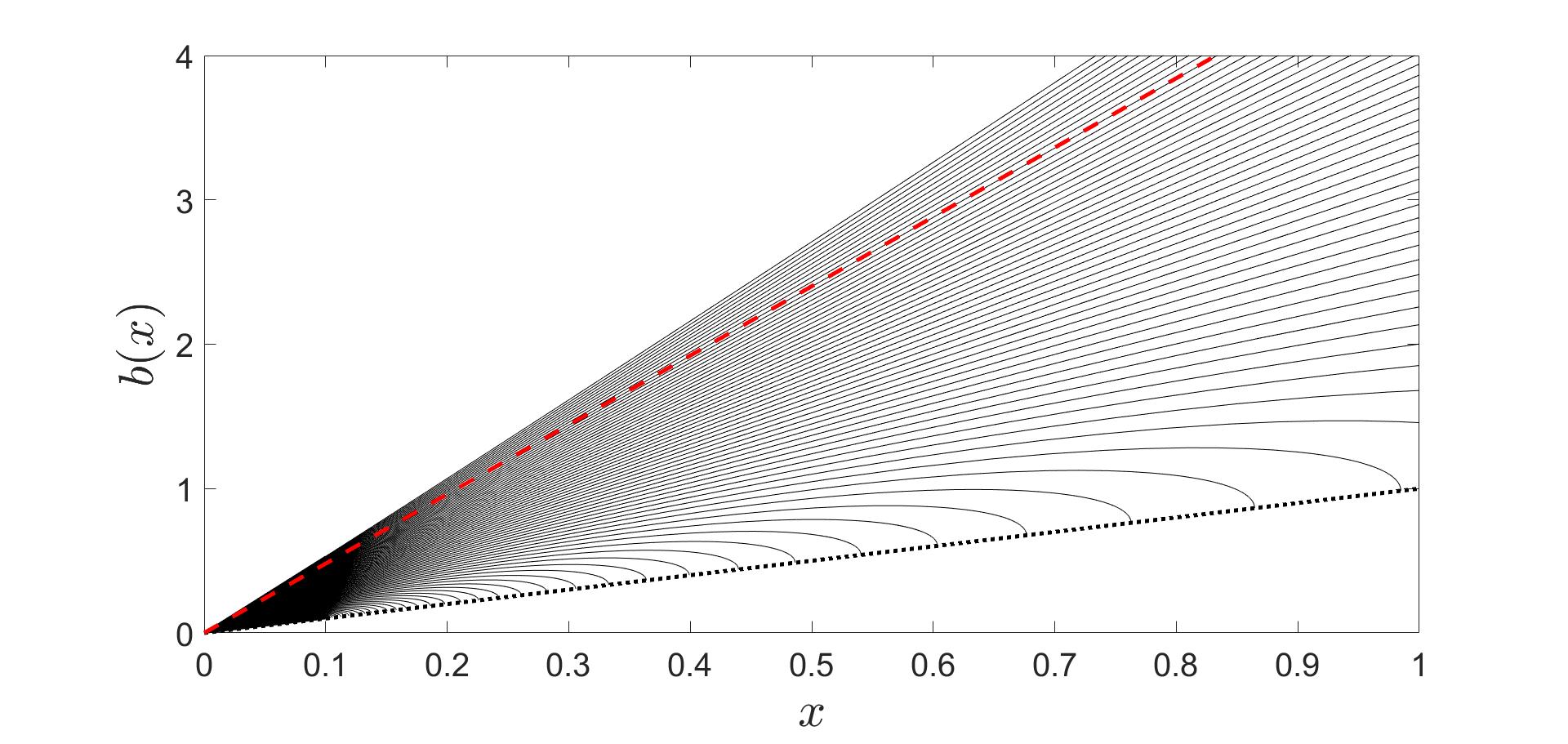}
\caption{Solutions to \eqref{ODEGBM} for varying initial data using $r=0.05$,  $\alpha= 0.04 $, and $\beta=0.3$. All solutions below $b^*(x)=Cx$ with $C\approx 4.80$ (dashed line) are concave
and hit the diagonal (dotted line), whereas all solutions above $b^*$ are convex. Solutions are obtained using MATLABs ODE-solver \texttt{ode15s}.}
\label{fig2}
\end{figure}

Solutions of the ODE \eqref{eq:ODE}  In this setting, the optimal dividend strategy is described by an affine boundary $b^*$ (dashed line).

It follows from Lemma~\ref{prop} that one solution of the ODE 
\begin{equation}
\label{ODEGBM}
b'(x)=2\frac{b(x)}{x}-\frac{\alpha b(x)}{rx}\frac{\gamma_2\varphi(x)\psi(b(x))-\gamma_1\varphi(b(x))\psi(x)}{\varphi(x)\psi(b(x))-\varphi(b(x))\psi(x)}
\end{equation}
for the boundary is given by $b^*(x)=Cx$. Since $C>1$, we have $b^*\in\B$. Moreover, if $b\in\B$ is another solution
with $b\leq b^*$, then $F(x,b(x)) = G(b(x)/x)\leq b(x)/x$ by Lemma~\ref{prop}, so
\[b''(x)=F_x(x,b(x))+F(x,b(x))F_y(x,b(x))\leq F_x(x,b(x))+\frac{b(x)}{x}F_y(x,b(x))=0\]
at all points since $xF_x(x,y)+yF_y(x,y)=0$. 
Consequently, any such $b$ is concave and, in the next paragraph, we will use such concavity to show that $b^*$
is the smallest solution that stays above the diagonal (Figure \ref{fig2}).

Assume that $b$ satisfies $b'(x)=F(x,b(x))$ and $x<b(x)\leq Cx$ for $x\in(0,\infty)$. First note that if
$b(x)\leq Ax$ for some $x\in(0,\infty)$, then $b'(y)\leq 0$ for $y\geq x$, so $b$ would not 
stay above the diagonal. Therefore we must have that $Ax\leq b(x)\leq Cx$ for
all $x$. By concavity, it then has an asymptote as $x\to\infty$, say $b(x)\sim Dx+E$ for some constants $D\in [A,C]$ and $E\geq 0$, with $b(x)\leq Dx+E$. However, then 
\[D\leq b'(x)=F(x,b(x))\leq F(x,Dx+E)\to F(x,Dx)\leq D\] 
as $x\to\infty$. Since the last inequality is strict if $D<C$, we must have $D=C$, and 
then also $E=0$ for the inequality $b\leq b^*$ to hold. This shows that $b=b^*$, 
so $b^*$ is the smallest element of $\B$.

The maximality principle thus suggests that the optimal dividend boundary is given by $b^*(x)=Cx$, where
$C$ is as above. Consequently, $\gamma_*=\inf\{t\geq 0:Y_t\leq \frac{M_t}{C}\vee x\}$, where $M_t:=\sup_{0\leq s\leq t}Y_s$
is the maximum process. Clearly, $\gamma_*\leq \inf\{n\geq 1:Y_n\leq\frac {Y_{n-1}}{C}\}<\infty$, $\P_{x,y}$-a.s.
Moreover, 
\[
\overline v^*(x,y)=\frac{y}{\gamma_2-\gamma_1}\Big(\frac{\gamma_2C^{-\gamma_1}}{1-\gamma_1}  \big(1-(x/y)^{1-\gamma_1}\big)
-\frac{\gamma_1C^{-\gamma_2}}{1-\gamma_2}  \big(1-(\frac{x}{y})^{1-\gamma_2}\big)\Big),
\]
and in particular, $v^*(0,y)=\overline v^*(y,y/C)+y/C=Ny$ for some constant $N$. Consequently, 
\[\E_{x,y}\left[\sup_{t\geq 0}\left\{e^{-rt}v^*(0,Y_t)\right\}\right]= N\E_{x,y}\left[\sup_{t\geq 0}\left\{e^{-rt}Y_t\right\}\right]<\infty
\]
since $e^{-rt}Y_t$ is a geometric Brownian motion with (strictly) negative drift. 

\begin{figure}
\includegraphics[scale=0.2]{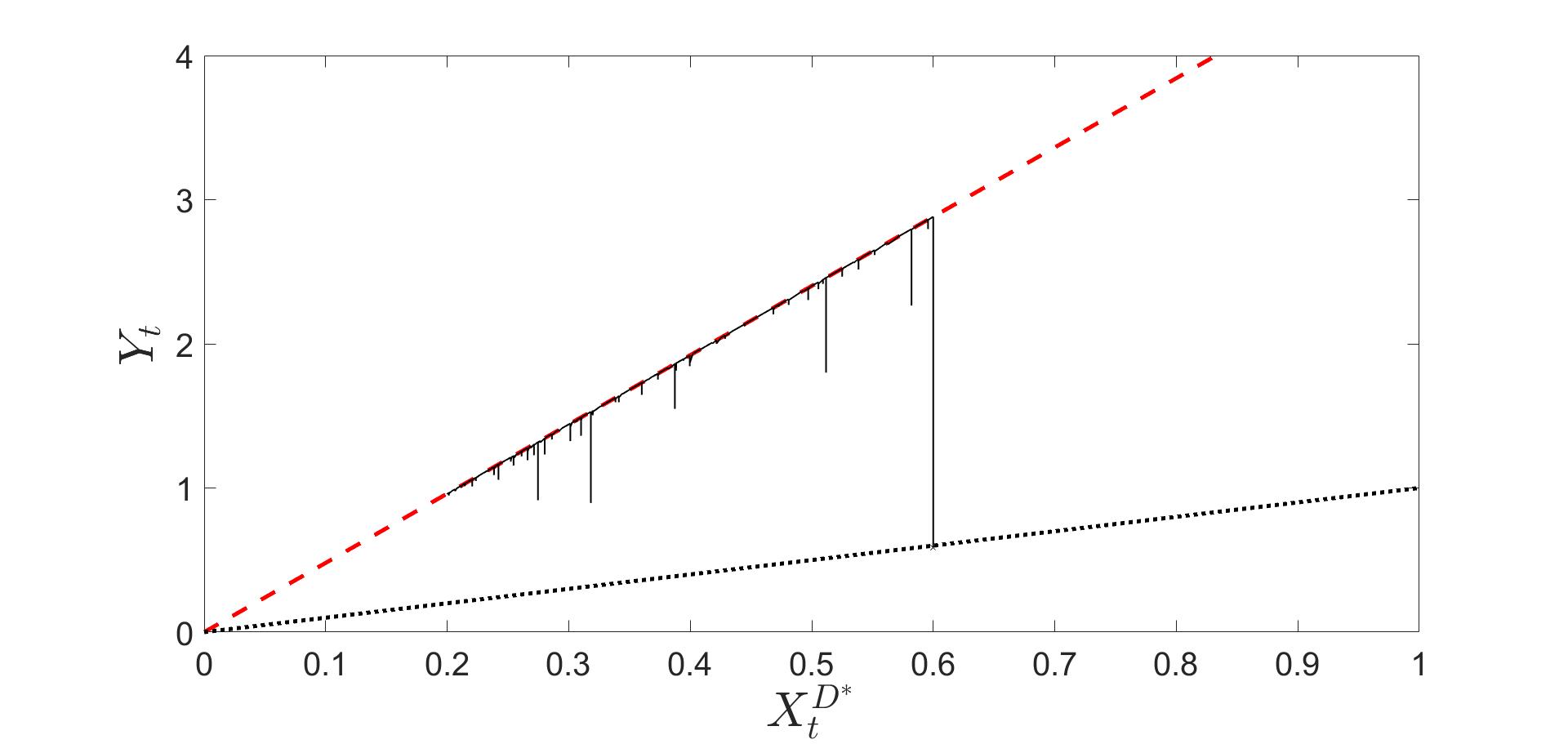}
\caption{One path of the process $(X_t^{D^*},Y_t)$ plotted until absorption in the diagonal (dotted line). Parameters are as in Figure \ref{fig2} and $(X^{D^\ast }_0, Y_0)=(0.2,0.2C)$.}
\label{path}
\end{figure}

It follows that the conditions of Theorem~\ref{thm:mainthm} are fulfilled, so an optimal
dividend boundary is given by the straight line $y=b^*(x)=Cx$. More specifically, 
the dividend strategy 
\[D^*_t=(\frac{1}{C}\sup_{0\leq s\leq t} Y_s-x)^+\]
is optimal in \eqref{eq:dividendvalue}. For an illustration of the optimally controlled path 
$(X^{D^*},Y)$, see Figure~\ref{path}.

\begin{remark}\label{rem:quality}
The structure of the optimal strategy we found here is different from the classical example with arithmetic Brownian motion studied in, e.g., \cite{AT,JPS,RS}. In that case dividends are paid optimally when the distance between the {\em pre}-dividend 
{equity capital}, $Y_t=y+\mu t+\sigma W_t$, and the total amount of dividends already paid out, $D_t$, is equal to a fixed constant $a^*$ (the optimal boundary), i.e., when $Y_t-D_t=a^*$. So, in that case, the optimal {\em distance} between $Y$ and $D$ is constant throughout the optimisation. In our problem formulation instead, the decision to pay dividends is determined based on a constant {\em ratio} between the process $Y$ and the process $D$. That is, dividends are paid out when $Y_t/X_t=C$. This scaling may be only {\em partially} read out of the logarithmic transformation that links geometric and arithmetic Brownian motion. Indeed, for the state variables $\tilde Y_t=\ln Y_t$ and $\tilde X_t=\ln X_t$ we  retrieve the optimality condition
\[
\tilde Y_t-\tilde X_t=a^*:=\ln C,
\]
and the absorption condition $Y\le X\iff \tilde{Y}\le\tilde{X}$. However, the logarithmic transformation of the expected payoff leads to
\[
\E\Big[\int_{[0,\gamma]}e^{-rt} \ud X_t\Big]=\E\Big[\int_{[0,\gamma]}e^{-rt} \ud e^{\tilde X_t}\Big]=\E\Big[\int_{[0,\gamma]}e^{-rt+\tilde X_t} \ud {\tilde X_t}\Big],
\]
where we notice that $\ud D=\ud X$ in the first expression and, for simplicity, we take $X$ with continuous trajectories in the change of variable formula. Therefore, after the logarithmic transformation the objective function is not the usual one from the classical dividend problem with arithmetic Brownian motion.
\end{remark}

\section*{Appendix}
\subsection*{Proof of Lemma \ref{lem:SK}}
Since $b\in\B$ then $b^{-1}$ is strictly increasing and continuous. Then the process $D^b$ is by definition continuous for $t>0$ and it has a single jump at time zero if $y>b(x)$. Moreover it is $(\F_t)$-adapted and non-decreasing and by continuity of paths $Y_{\gamma^b}=X^b_{\gamma^b}$ on $\{\gamma^b<\infty\}$. Therefore $D^b\in \A$. 

In order to prove that $(X^b,Y)$ solves the Skorokhod reflection problem we start by observing that, by construction, 
\[
X^{b}_t=x+D^b_t\ge b^{-1}(Y_t)\implies Y_t\le b(X^b_t),\quad\text{for all $t\ge 0$}.
\]
Fix $\omega\in\Omega$ and consider $s\ge 0$ such that $X^b_{s}(\omega)>b^{-1}\big(Y_s(\omega)\big)$. Then by definition of $X^b$ we have  \[
D^b_s(\omega)=\sup_{0\le  u\le s}\big(b^{-1}\big(Y_u(\omega)\big)-x\big)^+>b^{-1}\big(Y_s(\omega)\big)-x.
\]
Therefore, by continuity of $t\mapsto b^{-1}(Y_t(\omega))$
there exists $\eps_\omega>0$ such that $D^b_{s}(\omega)=D^b_{s'}(\omega)$ for all $s'\in [s,s+\eps_\omega]$, which implies $\ud D^b_s(\omega)=0$ as needed.\hfill$\square$

\subsection*{Proof of Lemma \ref{le:existencereflectedSDE}}
We first prove uniqueness. Recall that we have locally Lipschitz coefficients $(\mu,\sigma)$ with linear growth. In the notation of Bass \cite[Sec.\ 12]{Bass} we have $D=\{(x,y):y\ge x\}$, $\nu(x)=(-\frac{1}{\sqrt{2}},\frac{1}{\sqrt{2}})$ and $v(x)={\bf v}=(\tfrac{1}{2},1)$. Then, \cite[Thm.\ 12.4]{Bass} (see also the remark after the theorem) yields uniqueness and the strong Markov property of $(\hat X_{t\wedge\tau_n},\hat Y_{t\wedge\tau_n})_{t\ge 0}$ for $\tau_n:=\inf\{t\ge 0:\hat Y_t\ge n\}$ and any $n\in\mathbb N$. Linear growth of the coefficients implies $\tau_n\uparrow \infty$ a.s.\ as $n\to\infty$. Then we can obtain global uniqueness of a {strong Markov} solution $(\hat X_{t},\hat Y_{t})_{t\ge 0}$ by a standard limiting argument.

The proof of existence in Bass \cite{Bass} is given under an assumption of non-degeneracy of the reflecting diffusion which clearly fails in our case. For more general results Bass points to the classical paper by Lions and Sznitman \cite{LS}. Thanks to the special setting of our problem we can produce a simpler proof, which we include for completeness.

The main idea is to reduce the reflection problem to a classical problem for a reflecting Brownian motion. This can be achieved by a transformation via the scale function and a time-change. While this line of argument is canonical in the theory of one-dimensional diffusions, we believe the full proof might be difficult to find in the literature, hence we provide it here.

Let $S:[0,\infty)\to \mathbb R$ be the scale function associated to the coefficients $(\mu,\sigma)$ in the SDE for $\hat Y$ (see \eqref{eq:S}). Then, letting $\tilde Y_t:=S(\hat Y_t)$ and $\tilde X_t:= S(\hat X_t)$ and denoting $\tilde y:=S(y)$ and $\tilde x:= S(x)$, the dynamics of these two processes read
\begin{align*}
&\tilde Y_t=\tilde y+\int_0^t\tilde\sigma(\tilde Y_u)\ud W_u+\tilde A_t,\qquad \tilde X_t=\tilde x +\tfrac{1}{2}\tilde A_t,
\end{align*}
where $\tilde \sigma (y):=(S'\circ S^{-1})(y)(\sigma\circ S^{-1})(y)$ and 
\[
\tilde A_t:=\int_0^t(S'\circ S^{-1})(\tilde Y_u)\ud \hat A_t=\int_0^t(S'\circ S^{-1})(\tilde X_u)\ud \hat A_t,
\]
using that $\ud A_t$ is supported on $\{t:\tilde X_t=\tilde Y_t\}$.
Notice in particular that since $S$ is one-to-one, then \eqref{eq:reflectedSDE} admits an $(\F_t)_{t\ge 0}$-adapted solution if and only if the problem 
\begin{equation} \label{eq:reflectedSDE2}
\left\{\begin{array}{ll}
\tilde Y_t=\tilde y+\int_0^t\tilde\sigma(\tilde Y_u)\ud W_u+\tilde A_t,\qquad  \tilde X_t=\tilde x +\tfrac{1}{2}\tilde A_t, \\[+4pt]
\tilde Y_t \geq \tilde X_t\:\:\text{and}\:\: {\int_0^t 1_{\{\tilde Y_s>\tilde X_s \}} \ud\tilde A_s=0},\qquad \text{for all $t\ge 0$}\\
\end{array}
\right.
\end{equation}
admits one. Notice also that $\tilde X\in (S(x),\infty)$ by construction.

The next step removes the diffusion coefficient by a canonical time change. Indeed, the process 
\[
\tilde M_t:=\int_0^t \tilde \sigma (\tilde Y_u)\ud W_u
\]
is a continuous (local) martingale with quadratic variation 
\[
\langle \tilde M\rangle _t:=\int_0^t \tilde \sigma^2 (\tilde Y_u)\ud u.
\]
Since $\sigma(y)>0$ for $y>0$, then we have $\tilde \sigma(y)>0$ for $y>S(0)$ and the process $t\mapsto\langle \tilde M\rangle_t$ is strictly increasing.
Letting $\rho_t:=\inf\{s\ge 0:\langle \tilde M\rangle_s=t\}$ be the continuous inverse of $\langle\tilde M\rangle$ we have that $B_t:= \tilde M_{\rho_t}$ defines a continuous martingale with quadratic variation $\langle\tilde M\rangle_{\rho_t}=t$, hence $(B_t)_{t\ge 0}$ is a Brownian motion for the time-changed filtration $\check \F_t:=\F_{\rho_t}$ \cite[Thm.\ 3.3.16]{KSb}. Now, set $\check Y_t:=\tilde Y_{\rho_t}$, $\check X_t:=\tilde X_{\rho_t}$ and $\xi_t:=\tilde A_{\rho_t}$. Then 
\[
\xi_t:=\tilde A_{\rho_t}=\int_0^{\rho_t}1_{\{\tilde X_u=\tilde Y_u\}}\ud \tilde A_u
=\int_0^{\rho_t}1_{\{\check X_{\langle \tilde M\rangle _u}=\check Y_{\langle \tilde M\rangle _u}\}}\ud \xi_{\langle \tilde M\rangle _u}=\int_0^{t}1_{\{\check X_{s}=\check Y_{s}\}}\ud \xi_{s},
\] 
where the final equality holds by a simple change of variable, and \eqref{eq:reflectedSDE2} admits an $(\F_t)_{t\ge 0}$-adapted solution if and only if the problem below admits an $(\check \F_t)_{t\ge 0}$-adapted one:  
\begin{equation} \label{eq:reflectedSDE3}
\left\{\begin{array}{ll}
\check Y_t=\tilde y+B_t+\xi_t,\qquad  \check X_t=\tilde x +\tfrac{1}{2}\xi_t, \\[+4pt]
\check Y_t \geq \check X_t\:\:\text{and}\:\: {\int_0^t 1_{\{\check Y_s>\check X_s \}} \ud\xi_s=0},\qquad \text{for all $t\ge 0$}.\\
\end{array}
\right.
\end{equation}

Finally, letting $\check Z_t:= \check Y_t-\check X_t$ and $z:=\tilde y-\tilde x$ we have that \eqref{eq:reflectedSDE3} admits an $(\check \F_t)_{t\ge 0}$-adapted solution if and only if 
\begin{equation} \label{eq:reflectedSDE4}
\left\{\begin{array}{ll}
\check Z_t=z+B_t+\tfrac{1}{2}\xi_t, \\[+4pt]
\check Z_t \geq 0\:\:\text{and}\:\: {\int_0^t 1_{\{\check Z_s > 0 \}} \ud\xi_s=0},\qquad \text{for all $t\ge 0$}.\\
\end{array}
\right.
\end{equation}
The latter is just the classical Skorokhod reflection problem, whose solution is constructed explicitly by taking 
\[
\tfrac{1}{2}\xi_t= z\vee \sup_{0\le u\le t}(-B_u) - z.  
\] 
\hfill$\square$

%%%%%%%%%%%%%%%%%%%%%%%%%%%%%%%%%%%%%%

\end{document}